\documentclass[12,reqno]{amsart}
\usepackage{amsfonts}
\usepackage[active]{srcltx}
\usepackage{amsmath}
\usepackage{amssymb}
\usepackage{amsthm}
\usepackage{mathtools}
\usepackage{bm} 
\usepackage{empheq}

\usepackage{enumitem}
\usepackage{wasysym}
\usepackage{verbatim}
\usepackage{graphicx}
\usepackage[
bookmarks=true,         
bookmarksnumbered=true, 
colorlinks=true, pdfstartview=FitV, linkcolor=blue, citecolor=blue,
urlcolor=blue]{hyperref}
\usepackage{microtype}
\theoremstyle{plain}
\newtheorem{theorem}{Theorem}[section]

\newtheorem{lemma}[theorem]{Lemma}
\newtheorem{problem}[theorem]{Problem}
\newtheorem{proposition}[theorem]{Proposition}

\renewenvironment{proof}[1][Proof]{\textbf{#1.} }{\ \rule{0.5em}{0.5em} \par }
\theoremstyle{remark}

\theoremstyle{definition}
\newtheorem{remark}[theorem]{Remark}

\newtheorem{example}[theorem]{Example}
\newtheorem{hypothesis}[theorem]{Hypothesis}

\def\PP{\mathbb{P}}
\def\RR{\mathbb{R}}
\def\EE{\mathbb{E}}
\def\BB{\mathbb{B}}

\def\cB{{\mathcal B}}

\def\cF{{\mathcal F}}

\def\be{{\beta}}
\def\de{{\delta}}
\def\la{{\lambda}}
\def\si{{\sigma}}

\def\cT{{\mathcal T}}

\def\De{{\Delta}}

\def\cL{{\mathcal L}}
\def\Om{{\Omega}}
\def\om{{\omega}}
\def\al{{\alpha}}

\def\be{{\beta}}
\def\Ga{{\Gamma}}
\def\ga{{\gamma}}
\def\de{{\delta}}
\def\De{{\Delta}}

\def\si{{\sigma}}

\def\la{{\lambda}}
\def\La{{\Lambda}}
\def\vare{{\varepsilon}}
\def \eref#1{\hbox{(\ref{#1})}}
\def\th{{\theta}}
\def\th{{\theta}}
\def\si{{\sigma}}
\def\al{{\alpha}}

\def\eref#1{\hbox{(\ref{#1})}}

\newcommand{\HH}{\mathcal H}

\newcommand{\DD}{\mathbb D}

\let\Section=\section
\def\section{\setcounter{equation}{0}\Section}

\subjclass[2000]{60H10, 60G22,      65C30,  60H07,  60H05, 26A33.}

\begin{document}

\title[It\^o  SDE driven by fractional Brownian motion]{\bf  It\^o     stochastic differential equations
driven by fractional Brownian motions of Hurst parameter $H>1/2$
}

\author{  Yaozhong Hu }
\address{Department of Mathematics,\   University of Kansas
 \\  
 405 Snow Hall, \   Lawrence, \  KS 66045-2142 
 \\ 
Email: \ yhu@ku.edu
}
\date{}

\maketitle

\begin{abstract}This paper studies the existence and uniqueness of solution of 
 It\^o    type stochastic differential equation $dx(t)=b(t, x(t), \om)dt+\si(t,x(t), \om) d B(t)$,  where $B(t)$ is a fractional Brownian motion of Hurst parameter $H>1/2$ and $dB(t)$ is  the It\^o    differential 
defined 
  by using   Wick product or divergence operator. The coefficients $b$ and $\si$ are random and  anticipative. Using the relationship 
   between the It\^o       and pathwise integrals we first write the equation as a stochastic differential equation 
involving pathwise integral plus a Malliavin derivative term. To handle this Malliavin derivative term the  equation  is then further reduced to a system of characteristic  equations  without Malliavin derivative, which   is then   solved by a careful analysis of Picard iteration,  with a new technique to replace the Gr\"onwall  lemma which is no longer applicable.  The solution of this   system of characteristic equations  is  then  applied to solve the original It\^o   stochastic differential equation  up to a positive random  time. In  special  linear and quasilinear cases the global
solutions  are proved to exist uniquely. 
\end{abstract} 

\section{Introduction}
Let $T\in(0,\infty)$ be a given fixed number and let $\Om$ be the Banach
space of continuous real-valued functions   $f: [0, T]\rightarrow
\RR$ with the supremum  norm:
$\|f\|=\sup_{0\le t\le T} |f(t)|$.  
For any $t\in [0, T]$  define the coordinate mapping 
$B(t): \Om\rightarrow
\RR$ by $B(t)(\om)=\om(t)$. 
Let  $\PP=\PP^H$ be the probability measure on the Borel $\si$-algebra
$  \cF $ of $\Om$ such that $B=(B(t), 0\le t\le T)$  is a fractional
Brownian motion of Hurst parameter $H\in (0, 1)$.  Namely,  on the probability
space $(\Om, \cF, \PP)$, $B=(B(t), 0\le t\le T)$ is a centered
(mean $0$) Gaussian process of covariance given by
\[
\EE\left(B(t)B(s)\right)=\frac12 \left(t^{2H}+s^{2H}-|t-s|^{2H}\right)\,.
\]
Throughout the  paper, we consider the case $H>1/2$.   The natural
filtration generated by  $B(t)$ is denoted by $\cF_t$.  
For any $\be\in (0, H)$,  it is known that  
almost surely, $B(t)$ is H\"older 
continuous of exponent $\be$.   This means that there is a measurable subset
of $\Om$ of probability one  such that any element $\om$
in this set  $B(\cdot, \om)$ is H\"older continuous of exponent
$\be$.  We shall work on this subset of $\Om$ and with an abuse
of notation we shall    denote this subset still by $\Om$.    We also choose and
fix such a  $\be\in (1/2, H)$ throughout the paper.  

Fractional
Brownian motions have been received a great attention in recent
years. Stochastic integral,  It\^o    formula, and many other
basic  results have been established. The stochastic
differential equation of the form
\begin{equation}\label{e.1.1}
dx(t)=b(t, x(t))dt+\si(t, x(t))\de B(t)\,, \ \ 0\le t\le T\,,\quad
x(0)\quad \hbox{is given}\,,
\end{equation}
has been studied by many authors and has found many applications
in various fields, where $\de$ denotes the {\it
pathwise type integral}  defined by using Riemnan sum.  
Among
many references we refer to    \cite{BHOZ,hustochastics,misura} and in particular the references therein  for more
details.         Lyons and his collaborators'  work on rough path
analysis is  a  powerful tool  in analyzing  this type of equation (see
\cite{FV,lyonsqian}   and the references therein).

However,  in the case when $B$ is the Brownian motion,
the most studied equation is of It\^o    type. Namely,
 in  \eref{e.1.1}   It\^o    stochastic differential is used instead of the
 pathwise one.  There are many reasons for the use of  
 It\^o    stochastic differential  in classical Brownian motion case. 
 One reason is from the modeling point of view.  If one uses \eqref{e.1.1}
 to model the state of a certain system, 
then the term $b(t, x(t))$  represents {\it all} the ``{\it mean rate of change}"   
of the system  and the term 
  $\si(t, x(t))\de B(t)$ is  the ``{\it random
perturbation}",  which has a zero mean contribution.   

When we use  stochastic differential equations
driven by   fractional Brownian motion to model natural or social system,
we also wish to separate the two parts: the part $b(t, x(t))$ represents {\it all}  the mean rate of change
and the part $\si(t, x(t))\de B(t)$ is merely the random 
perturbation,  which should have a mean $0$.   
  In another word, 
it is natural to require  the {\it mean}
  of  $\si(t, x(t))\de B(t)$ in \eqref{e.1.1} to be  zero. 
On the other hand,   it is   well-known from the work of  \cite{duncan, humams} that  if $H\not=1/2$,    then the pathwise type
stochastic integral with respect to fractional Brownian motion may not be of zero mean.  Namely, it is possible that
$\EE \left[\int_0^T \si(t, x(t))\de B(t)\right] \not= 0$.  Motivated by this
phenomenon,  an  It\^o    type stochastic integral $\int_0^T \si(t,
x(t))d B(t)$ is introduced with the use of Wick product in     
\cite{duncan, huoksendal}    (see    \cite{BHOZ,humams}
and   the
references therein).  This integral has the property that the expectation 
 $\EE\left[
\int_0^T \si(t, x(t))d B(t)\right]$ is always equal to zero. This motivates 
to replace the pathwise integral in \eqref{e.1.1} by the It\^o    one.
In other words,  we are led  to consider the  following {\it It\^o     stochastic
differential equation}
\begin{equation}\label{e.ito-equation}
dx(t)=b(t, x(t))dt+\si(t, x(t))d  B(t)\,, \ \ 0\le t\le T\,,\quad
x(0)\quad \hbox{is given}\,,
\end{equation}
where $dB(t)$ denotes the {\it It\^o    type stochastic differential}
  (divergence type integral) defined in \cite{duncan} (see  also  \cite{BHOZ, humams, huoksendal}  and references therein),    $b$
and $\si$ are two real-valued functions from $ [0, T]\times \RR$ to
$\RR$ satisfying
 some conditions that will be made precise  later (we shall allow them to be random).
To solve the above equation \eqref{e.ito-equation}, 
a natural  approach  to try      is
   the Picard iteration.  To explain the difficulty let us define $x_n(t)$ 
by the following recursive formula (naive Picard iteration):
\begin{equation}\label{e.1.3}
x_n(t)  =x(0)+\int_0^t b(s, x_{n-1}(s)  )dt+\int_0^t \si(s, x_{n-1}(s)  )d
B(s)\,, \quad  0\le t\le T\,,  
\end{equation}
 where  $  n=1, 2, \cdots
$ and $x_0(t) :=x(0)$ for all $0\le t\le T$.
Consider the above stochastic integral  term on the right hand side.
An  It\^o    isometry formula states   that
\begin{eqnarray*}
&&\EE \left(\int_0^t \si(s, x_{n-1}(s) )d  B(s)\right)^2\\ 
 &=& \EE
\left\{  \int_0^t\int_0^t \phi(u, v)  \si(u, x_{n-1}(u) ) \si(v,
x_{n-1}(v) )  dudv
\right.\\
&&\quad \left. + \int_0^t\int_0^t     \si_x (u, x_{n-1}(u) ) \si_x (v, x_{n-1}(v) ) 
\DD^\phi_v x_{n-1}(u)    \DD^\phi_u x_{n-1}(v)
dudv   \right\} \,,
\end{eqnarray*}
where 
$\phi(u,v)=H(2H-1)|u-v|^{2H-2}$,  $\si_x(t,x)$ denotes the partial derivative of $\si(t,x)$ with respect to $x$ and $\DD^\phi_u$ is the Malliavin derivative (see forthcoming  
definition  \eqref{e.def-Dphi}  in next section).  
From this identity one sees  that   to bound  the $L^2$ norm of $x_n $ one has to 
use   the  $L^2$ norm of $x_{n-1} $ plus the $L^2$ norm of the Malliavin
derivative of $x_{n-1} $.  In a similar way to bound the Malliavin
derivative   one has
 to use the  second order Malliavin derivative,  and so on.   
 Thus, we see that  the naive  Picard iteration approximation 
 cannot be applied   to study the It\^o  stochastic differential
  equation \eref{e.ito-equation}.      

We shall use a different approach to study \eqref{e.ito-equation}.  
To explain this approach  we   first use    the relationship  between pathwise
and It\^o    stochastic integrals (established for example in   \cite[Theorem 3.12]{duncan}.  See also \cite{BHOZ} and \cite{humams}) 
to  write the equation \eref{e.ito-equation} as
\begin{eqnarray} 
x(t)&=&x(0)+\int_0^t b(s, x(s))ds+\int_0^t
\si(s, x(s))\de  B(s)\nonumber\\
&&\qquad -\int_0^t   \si_x(s, x(s)) \DD^\phi_s  x(s) ds\,,
 \ \ 0\le t\le T\,.\label{e.1.4}
\end{eqnarray}
Thus,   the equation \eref{e.ito-equation} is reduced  
to an equation involved the pathwise integral plus a Malliavin derivative term.
 To understand the character
of this equation, we consider heuristically the dependence on the random element
$\om\in \Om$ of the random  variable 
$x(t, \om)$ as a function of infinitely many variables (defined on $\Om$). 
We write it formally as 
$x(t, \om)=u(t, \tilde \ell_1, \cdots, \tilde \ell_n , \cdots)$,  where $u(t, x_1, x_2, \cdots)$ is a function of infinitely many variable and 
$\ell_1, \cdots \ell_, \cdots$ are smooth deterministic functions such that
$\langle \ell_i\,, \ell_j\rangle_{\HH_\phi}=\begin{cases}1& \hbox{when $i=j$} \\
0&\hbox{otherwise}\end{cases}$ (see 
 the forthcoming   definition \eqref{e.def-Hphi} for  the Hilbert space $\HH_\phi$.
 We usually assume that $\left\{\ell_1, \cdots \ell_, \cdots\right\}$
 to be an orthonormal basis of $\HH_\phi$)
and $\tilde \ell_i=\int_0^T\ell_i(s) dB(s)$.  Thus $\DD_s^\phi x(s)
=\sum_{i=1}^\infty \phi_i(s) \frac{\partial u}{\partial x_i}  (s, x_1, x_2, \cdots) $ with
$\phi_i(s)=\int_0^T \phi(s, r) \ell_i(r) dr$.  
With the above notations,  the equation \eqref{e.1.4} can be written as (we omit the explicit dependence of $u$ on 
$(x_1, x_2, \cdots)$)
\begin{eqnarray} 
u(t)&=&u(0)+\int_0^t b(s, u(s))ds+\int_0^t
\si(s, u(s))\de  B(s)\nonumber\\
&&\qquad -\sum_{i=1}^\infty \int_0^t \phi_i(s) 
\si_x(s, u(s))  \frac{\partial u}{\partial x_i} ds\,,
 \ \ 0\le t\le T\,.\label{e.1.5}
\end{eqnarray}
%
%
This   is a 
first order hyperbolic partial differential equation driven by fractional
Brownian  motion for a function of infinitely many variables.  We shall use the idea of characteristic curve approach 
from the theory of  the first order (finitely many variables) 
hyperbolic equations  (see for example \cite{li, serre1, serre2}).  But since the classical theory is not directly applicable here
we need to find the characteristic curve and prove the existence and uniqueness of the
solution.    Let us also point out that 
  stochastic hyperbolic equations has also been studied by several  authors 
(see e.g.  \cite{Kunita}),   which is different from  ours.  

Now we explain our approach  to solve \eqref{e.1.4}.   
First, we construct  the following
coupled system of characteristic equations:
\begin{equation}
\begin{cases}\Gamma(t)=\om+\int_0^t    \si _x(s, z(s) )\int_0^\cdot \phi(s,
u)du ds\,;\\ \\
z(t)= \eta(\om ) +\int_0^t   b(s, z(s) )ds
+\int_0^t \si(s, z(s) )\de B(s)\\ 
\qquad\qquad  +\int_0^t\int_0^s  \si(s, z(s) )   \si_x(u, z(u) ) \phi(s, u) duds\,,
\end{cases}  \label{e.1.6} 
\end{equation}
where $\si_x$ denotes the partial derivative with respect to $x$,  $\Ga(t):\Om \rightarrow \Om$ and  $z(t):\Om \rightarrow \RR$.  
This system of equations  comes from the characteristic curve equation for the  first order hyperbolic equation \eqref{e.1.5}  of  a function of infinitely many variables 
(see \cite{li,serre1,serre2}). 
For this system of characteristic  equations,   we shall show the following
statements.   
\begin{enumerate}
\item[(i)]  We use the Picard iteration approach to show that the 
above system of equations \eqref{e.1.6}  has a unique solution. Since  the fractional Brownian motion  $B$ is not differentiable,  the powerful Gr\"onwall  lemma
cannot no longer be used.  Additional effort   is needed to solve the corresponding 
\eqref{e.1.6}.  We 
use a different contraction argument, 
  presented in Section 4.   We call this approach {\it fractional Picard iteration}  and hope that this general  contraction principle   may also be useful 
in solving other equations involving H\"older continuous controls.  Let us point out that   \eqref{e.1.6}  has a global solution. 
\item[(ii)]  We show in Section 5 that $\Ga(t):\Om\rightarrow \Om$ defined  by 
\eqref{e.1.6}  has an inverse $\La(t)$  when $t$ is sufficiently
small (smaller than a positive random constant)  and $x(t, \om)=z(t, \La(t, \om))$
satisfies \eqref{e.1.4} (or \eqref{e.ito-equation}).   
To this end  we need to use a new It\^o    formula which is quite interesting itself.
This new It\^o    formula  is presented in Section 3.    
\item[(iii)] For general nonlinear equation \eqref{e.ito-equation}  we can only solve the equation up to a positive (random)
  time.  This is because the inverse $\La(t)$  of $\Ga(t):\Om\rightarrow 
\Om$ exists only up to  some  random 
 time (see one example given in 
Section 5).  However,
 for linear or quasilinear equation  we can solve the equation
for all  time $t\ge 0$.  In particular, in one dimensional linear case, we can 
find the explicit solution.   This is done in Section 6. 
\end{enumerate}
%
%
%

For notational simplicity we only discuss one dimensional equation.  The system
of several equations can be handled in a similar way.  It is only notationally more complex. 
 On the other hand, 
our approach works for more general random anticipative  coefficients   with 
general anticipative random initial conditions. We present our 
work in this generality.  This means we shall study a slightly more general 
equation (see \eqref{e.5.6} in Section 5) instead of 
\eqref{e.1.6}.  There has been an intensive study on anticipative
   stochastic differential equations 
by using   anticipative calculus (see \cite{Bu92, Bu94, nualart}). 
We hope our work can   shed some lights  to  this topic as well. 
Some preliminary results  
are presented in Section 2 and some notations  used in this paper are 
also fixed there.

\section{Preliminary}
\subsection{Fractional integrals and derivatives}

Denote  $\left( -1\right) ^{\alpha }=e^{i\pi \alpha  }$ and   $ 
\Gamma \left( \alpha \right) =\int_{0}^{\infty }r^{\alpha -1}e^{-r}dr$. 
Let $a,b\in \mathbb{R}$ with $a<b$  and let $f\in L^{1}\left( a,b\right) $ and $%
\alpha >0.$  The left-sided and right-sided fractional Riemann-Liouville
integrals of $f$ 
are defined  
by 
\[
I_{a+}^{\alpha }f\left( t\right) =\frac{1}{\Gamma \left( \alpha \right) }%
\int_{a}^{t}\left( t-s\right) ^{\alpha -1}f\left( s\right) ds
\]%
and
\[
I_{b-}^{\alpha }f\left( t\right) =\frac{\left( -1\right) ^{-\alpha }}{\Gamma
\left( \alpha \right) }\int_{t}^{b}\left( s-t\right) ^{\alpha -1}f\left(
s\right) ds,
\]%
respectively if the above integrals exist,  where $a\le t\le b$.  
The Weyl derivatives are defined as  (if the integrals exist)
\begin{equation}
D_{a+}^{\alpha }f\left( t\right) =\frac{1}{\Gamma \left( 1-\alpha \right) }%
\left( \frac{f\left( t\right) }{\left( t-a\right) ^{\alpha }}+\alpha
\int_{a}^{t}\frac{f\left( t\right) -f\left( s\right) }{\left( t-s\right)
^{\alpha +1}}ds\right)  \label{e.weyl-der-left}
\end{equation}%
and
\begin{equation}
D_{b-}^{\alpha }f\left( t\right) =\frac{\left( -1\right) ^{\alpha }}{\Gamma
\left( 1-\alpha \right) }\left( \frac{f\left( t\right) }{\left( b-t\right)
^{\alpha }}+\alpha \int_{t}^{b}\frac{f\left( t\right) -f\left( s\right) }{%
\left( s-t\right) ^{\alpha +1}}ds\right)\,.   \label{e.weyl-der-right}
\end{equation}%

For any $\lambda \in (0,1)$,   denote by $C^{\lambda }(a,b)$ the space of $%
\lambda $-H\"{o}lder continuous functions on the interval $[a,b]$. We will
make use of the notations%
\[
\left\| x\right\| _{a,b,\beta }=\sup_{a\leq \theta <r\leq b}\frac{%
|x_{r}-x_{\theta }|}{|r-\theta |^{\beta }},
\]%
and%
\[
\Vert x||_{a,b}=\sup_{a\leq r\leq b}|x_{r}|,
\]%
where $x:\mathbb{R}^{d}\rightarrow \mathbb{R}$ is a given continuous
function. 
%
%
We refer to \cite{samko} for more details on 
fractional integrals and   derivatives.

  Let $\pi: a=t_0<t_1<\cdots<t_{n-1}<t_n=b$ be a partition
of $[a, b]$ and denote  $|\pi|=\max_{0\le i\le n-1} (t_{i+1}-t_i)$.  
Assume  that $f\in C^{\lambda }(a,b)$ and $g\in C^{\mu }(a,b)$ with $\lambda
+\mu >1$. For these two functions, we  define the Riemann sum 
$\displaystyle S_\pi(f|g)=\sum_{i=0}^{n-1} f(t_i)(g(t_{i+1})-g(t_i))$.  
 From a  classical result of  Young \cite{young},  we know that as $|\pi|\rightarrow 0$,
the limit of $S_\pi(f|g)$  exists and is  called  the
Riemann-Stieltjes integral 
\[
\int_{a}^{b}fdg=\lim_{|\pi|\rightarrow 0}  S_\pi(f|g)\,. 
\]
 We also have  the following
proposition.  
\begin{proposition}
\label{p.integration-by-parts} Suppose that $f\in C^{\lambda }(a,b)$ and $g\in C^{\mu }(a,b)$
with $\lambda +\mu >1$. Let $1-\mu<\al<{\lambda } $. Then
the Riemann Stieltjes integral $\int_{a}^{b}fdg$ exists and it can be
expressed as%
\begin{equation}
\int_{a}^{b}fdg=(-1)^{\alpha }\int_{a}^{b}D_{a+}^{\alpha }f\left( t\right)
D_{b-}^{1-\alpha }g_{b-}\left( t\right) dt,  \label{1.8}
\end{equation}%
where $g_{b-}\left( t\right) =g\left( t\right) -g\left( b\right) $.
\end{proposition}
We also note that  in   the convergent Riemann sum,  we can also use
\[
\tilde S_\pi(f|g)=\sum_{i=0}^{n-1} f(\xi_i)(g(t_{i+1})-g(t_j))\,,
\]
where $\xi_i$ is any point in $[t_i, t_{i+1}]$.  

Let $\Om$, $\HH$ be two separable Banach spaces such that $\HH$ is continuously 
embedded in $\Om$.  Let $\BB$ be another separable Banach space.
A mapping $F:\Om\rightarrow \BB$ is called $\HH$-differentiable 
if there is a bounded linear mapping from $\HH$ to $\BB$ (if such mapping exists,
then it is unique and is denoted
by  $\DD F(\om)$) such that
\begin{equation}
\frac{F(\om+\vare h)-F(\om)}{\vare}=\DD F(\om)(h)+ o(\vare)\quad \hbox{as $\vare\rightarrow 0$}\,; \quad \forall \ \om\in \Om,\
 h\in \HH\,.  
\end{equation}
We  can also considered  $\DD F(\om)$ as an element in the
tensor product space $  \BB\otimes \HH'$,  where $\HH' $
is the dual of $\HH$.  
The directional derivative $\DD_h F$ for any 
direction $h\in \HH$  is defined as $\DD_hF(\om)=
\DD F(\om)(h)=\langle \DD F(\om), h\rangle_{\HH',\HH}$.   To simplify notation we also write
$\DD_hF(\om)=\DD F(\om) h$.  
If    $\HH$   and $\BB$ are Banach spaces  of real functions, and 
if there is a function $g(s, \om)$ such that  
$\DD_hF(\om)=
\DD F(\om)(h) =\int_0^T f(s, \om)h(s)ds, \ \forall h\in \HH$, 
then we denote $\DD_sF(\om)=g(s, \om)$.  

It is easy to see that we have the following chain rule.
If $F:\Om\rightarrow \BB_1$ is $\HH_1$-differentiable
and $G:\BB_1\rightarrow \BB_2$ is $\HH_2$-differentiable such that 
$\DD F(\om)$ is a bounded mapping from $\HH_1$ to $\HH_2$,  then
$G\circ F$ is also $\HH_1$-differentiable and 
\begin{equation}
\DD G(F(\om))=(\DD G)(F(\om))\circ \DD F(\om)\,.
\end{equation}

%

For any function $f(t)=f(t, \om)$, where $(t, \om)\in [0, T]\times \Om $,
which is H\"older continuous with respect to $t$ of exponent $\mu>1-H$,
  by Proposition \ref{p.integration-by-parts},
we can define  the pathwise   integral $ \int_0^t f(s)\de B(s)$
as the (pathwise) limit as $|\pi|=\max_{0\le i\le n-1}
(t_{i+1}-t_i)\rightarrow 0$  of the following Riemann sum
\begin{equation}
S_\pi(f) =\sum_{i=0}^{n-1} f(s_k) (B(s_{k+1})-B(s_k))\,, 
\label{e.riemann-sum-left} 
\end{equation}
where $\pi: 0=s_0<s_1<\cdots<t_{n_{n-1}}<t_n=t$ is a partition of $[0, t]$. 
This integral can also be given  by 
\[
 \int_0^t f(s)\de B(s)=\int_0^t D^{1-\al} _{0+} f(s) D_{t-}^{ \al} B_{t-}(s) ds\,,
\]
where $\al$ satisfies $ 1-\mu<\al<H$.  

If  $f$ is H\"older continuous of exponent greater than $1-H$ and if $g$ is continuous,  then 
  $\eta(t)=\eta(0)+\int_0^t f(s) \de B(s)+\int_0^t g(s) ds$ is well-defined.  For any   continuous  function $F$ on $[0, T]\times \RR$, which is continuously differentiable 
  in $t$ and twice continuously in $x$ 
 we  have the  following It\^o    formula:  
\begin{eqnarray}
F(t, \eta(t))
&=& F(0, \eta(0))
 +\int_0^t \left[ \frac{\partial }{\partial s} F(s, \eta(s)) +\frac{\partial }{\partial x}F (s, \eta(s)) 
g(s) \right] ds\nonumber\\
&&\qquad +
\int_0^t \frac{\partial }{\partial x}F (s, \eta(s))  f(s) \de B(s) \,. 
\label{e.pathwise-ito}
\end{eqnarray}
[see for example  \cite{hustochastics} and references therein.] \  
We also notice that if $f:[0, T]\times \Om\rightarrow \RR$ is H\"older continuous of exponent $\mu>1-H$, then in the Riemann sum \eqref{e.riemann-sum-left}  the left point  $s_k$  can be replaced by  any points $\xi_k$ in the subinterval.  Namely,
\begin{equation}
\int_0^T f(s) \de B_s=\lim_{|\pi|\rightarrow 0} 
\sum_{i=0}^{n-1} f(\xi_k) (B(s_{k+1})-B(s_k))\,,
\label{e.def-pathwise-any-point} 
\end{equation}
where $\xi_k $ is  any point in  $[s_k, s_{k+1}]$.

In the stochastic analysis of fractional Brownian motions of Hurst parameter 
$H>1/2$, usually, we take the above Banach space 
 $\HH$ to be the reproducing kernel  Hilbert space
$\HH_\phi$:
\begin{equation}
\HH_{\phi}=\left\{ f:[0, T]\rightarrow\RR\,, \|f\|_{\HH_{\phi}}^2=
\int_0^T\int_0^T f(u) f(v)\phi(u-v) dudv<\infty\right\}\,,
\label{e.def-Hphi}
\end{equation}
which is the completion of the space of smooth functions on $[0, T]$ with respect to the norm $\|\cdot\|_{\HH_\phi}$,  
where 
\[
\phi(u):=H(2H-1) |u|^{2H-2}\,. 
\] 
The element in $\HH_\phi$ may be    generalized  function (distribution)
although we still write $f:[0, T]\rightarrow \RR$ in \eqref{e.def-Hphi}. 
We can define $\DD_s F(\om)$ as usual  and we 
denote
\begin{equation}\label{e.def-Dphi}
\DD^\phi_t F(\om) =\int_0^T  \phi(t,s)  \DD_s F(\om) ds\,. 
\end{equation} 
The expectation $\EE \int_a^bf(s) \de B(s)$ may generally not be zero.
In   \cite{duncan} (see also \cite{BHOZ, humams}) we introduce an 
  It\^o     stochastic integral by using the Wick product.  
  We also established a relationship between pathwise and It\^o   integrals. Here,  we can use this relationship to define It\^o      integral as 
\begin{equation}\label{e.ito-pathwise-relation}
\int_0^T f(t)dB(t)=\int_0^T f(t)\de B(t)-\int_0^T \DD^\phi _t f(t)dt 
\end{equation}
if $f$ is H\"older continuous of exponent $\mu>1-H$ and 
$ \DD^\phi _s f(s)$ exists and is integrable.  It is easy to see that
$\EE \left( \int_0^T f(t)dB(t)\right)=0$.  

The It\^o    formula and many other results for It\^o     stochastic 
integral have been established.  Here, we explain that the It\^o    formula 
for It\^o     integral can also be obtained   from 
\eqref{e.pathwise-ito} and \eqref{e.ito-pathwise-relation}.

\begin{proposition}\label{p.2.2} 
Let 
\[
\eta(t)=\eta+\int_0^t f(s) dB_s+\int_0^t g(s) ds\,,
\]
where $f$ is H\"older continuous with exponent greater than $1-H$
and $g$ is continuous. 
Assume that
$\DD^\phi_s f(s)$ exists and is a continuous function of $s$. Let $F:[0, T]\times \RR
\rightarrow \RR$ be   continuously differentiable in 
$t$ and twice continuously  differentiable in $x$.  Then 
\begin{eqnarray}
F(t, \eta(t))
&=&F(0, \eta_0)  
+\int_0^t  \frac{\partial F}{\partial x} (s, \eta(s))f(s)d B(s)
\label{e.ito-ito-formula} \\
&&  +\int_0^t   \left[  \frac{\partial F}{\partial s} (s, \eta(s))
+ \frac{\partial  F}{\partial x } (s, \eta(s)) 
 g(s)   + \frac{\partial^2 F}{\partial x^2} (s, \eta(s)) f(s)  \DD_s ^\phi \eta(s) \right] ds\,.  
 \nonumber 
  \end{eqnarray}
\end{proposition}
\begin{proof} We briefly sketch the proof. 
First,  by \eref{e.ito-pathwise-relation}  we see 
\[
\eta(t)=\eta_0+\int_0^t  f(s) \de B_s+\int_0^t \left[g(s)-\DD^\phi_s f(s)  \right]  ds\,.
\] 
From
\eref{e.pathwise-ito}  it follows that
\begin{eqnarray*}
F(t, \eta(t))
&=&F(0, \eta_0)+
\int_0^t \frac{\partial F}{\partial s} (s, \eta(s))ds
+\int_0^t  \frac{\partial F}{\partial x} (s, \eta(s))f(s)\de  B(s)
\nonumber\\
&&\qquad +\int_0^t   \frac{\partial F}{\partial x} (s, \eta(s)) 
\left[ g(s)- \DD^\phi_s f(s)\right] ds\\
&=&F(0, \eta_0)+
\int_0^t \frac{\partial F}{\partial s} (s, \eta(s))ds
+\int_0^t  \frac{\partial F}{\partial x} (s, \eta(s))f(s)d B(s)
\nonumber\\
&&\qquad +\int_0^t   \left\{  \frac{\partial F}{\partial x} (s, \eta(s)) 
  \left[g(s)- \DD^\phi_s f(s)\right]
  +\DD^\phi_s \left(\frac{\partial F}{\partial x} (s, \eta(s))f(s)
  \right)\right\} ds\,.  
 \end{eqnarray*}
 This is simplified to \eqref{e.ito-ito-formula}.
 \end{proof}
%
%
%
%

\section{It\^o     formulas}

Denote $\cT=[0, T]$.
If $X(t)=\eta+\int_0^t f(s)ds+\int_0^t g(s) \de  B(s)$ and if $F$ is a function
from $\cT\times \RR$ to $\RR$,  then an  It\^o     formula for $F(t, X(t))$ is
given by \eqref{e.pathwise-ito}, or \eqref{e.ito-ito-formula} 
if the integral is It\^o     type 
(see  also \cite{young, zahle, hustochastics} 
and in particular the  references therein).  
However, to show the existence and uniqueness of the solution 
to   It\^o    
  stochastic differential equation    we need  an It\^o     formula
of  the following form:
If $X$ is as above, $\Ga: \cT\times \Om\rightarrow \Om$ 
and $F:  \cT\times \RR\times \Om$ to $\RR$,  we want to find an It\^o     formula 
for $F(t, X(t), \Ga(t))$.  Here and in what follows, we omit the explicit dependence on $\om$
when it is clear.

\begin{lemma}\label{t.3.1}  
Let $h(t, u, \om)$, $(t,u,\om)\in \cT^2\times  \Om$ be a continuous function of $t$ and $u$.  
Define  a family  
of  nonlinear  transforms  from $\Om$ to $\Om$ by 
\begin{equation}
\Ga(t, \om)=\om+ \int_0^\cdot   h(t, u, \om) du\,,\quad 
t\in \cT\,. 
\end{equation}
Let  $f:\cT\times \Om\rightarrow \RR$  be  measurable such  that 
for any $\om\in \Om$,   $f:\cT  \rightarrow \RR$ is H\"older continuous 
of order greater than $1-H$ so that  $F=\int_0^T f(s) \de B(s)$ is well-defined.
We have 
\begin{eqnarray}
F\circ \Ga(t, \om)
&=&\int_0^T f(s, \Ga(t, \om) )\de B(s) +\int_0^T   f(s, \Ga(t, \om) )  h(t,s, \om)  ds\,.
\nonumber\\ 
\label{e.3.2} 
\end{eqnarray}
\end{lemma}

\begin{proof}  By a limiting argument we may assume that
$f$ is of the form
\[
f(t, \om)=\sum_{k=0}^{n-1} a_k (\om) I_{[t_k, t_{k+1})}(t)\,,
\]
where $0=t_0<t_1<\cdots<t_{n-1}<t_n=T$ is a partition of the
interval $[0, T]$.  Thus
\begin{eqnarray*}
F(\om)
&=& \sum_{k=0}^{n-1} a_k(\om)\left[B({t_{k+1}}, \om)-B({t_k}, \om)\right] \\
&=& \sum_{k=0}^{n-1} a_k (\om) \left[\om(t_{k+1})-\om(t_k) \right]\,.
\end{eqnarray*}
Thus
\begin{eqnarray*}
F(\Ga(t,\om))
&=& \sum_{k=0}^{n-1} a_k(\Ga(t,\om) ) \Bigg\{ \om({t_{k+1}}) +
\int_0^{t_{k+1}} h(t, s, \om)ds  -\om({t_k})
 -\int_0^{t_{k}} h(t, s, \om)ds  \Bigg\}  \\
&=& \sum_{k=0}^{n-1} \left\{a_k(\Ga(t,\om) ) ( \om({t_{k+1}})-\om({t_k})) + a_k(\Ga(t,\om) ) 
\int_{t_k} ^{t_{k+1}} h(t, s, \om)ds   \right\}
\\
&=&\int_0^T f(s, \Ga(t, \om) )\de B(s)+\int_0^T f(s, \Ga(t, \om) )h(t, s, \om)ds    \,,
\end{eqnarray*}
which is \eqref{e.3.2}.  
\end{proof}  
 
Let $\HH\subseteq \Om$ be a Banach space continuously embedded 
in $\Om$. Now we state our new It\^o     formula. 
\begin{theorem}\label{t.3.2}  Let   measurable functions $\eta: \RR\times \Om\rightarrow \RR $,
$f_0, f_1: \cT\times \RR\times \Om\rightarrow \RR $, $g_0, g_1:\cT\times \Om\rightarrow \RR$ satisfy
\begin{equation}
\begin{cases}
\hbox{ $f_0(s, x, \om)
$ and $g_0(s,   \om)$ are continuous in $s\in \cT$} \,;\nonumber\\
  \hbox{$f_1(s,x,\om)$ and $g_1(s,\om)$  are  H\"older continuous}\nonumber\\
\qquad  \qquad\quad \hbox{ with respect to $s$
of order greater than $1-H$}\,;\nonumber\\
  \hbox{$f_1(s,x,\om)$ is Lipschitz in $x$}\,. 
\end{cases}
\end{equation}
Define
\begin{eqnarray}
F(t,x,\om)
&=& \eta(x, \om)+\int_0^t f_0(s, x, \om) ds+\int_0^t f_1(s, x, \om) \de B(s)
\end{eqnarray} 
and 
\begin{eqnarray}
G(t, \om) 
&=&\xi(\om)+\int_0^tg_0(s,   \om) ds+\int_0^t g_1(s,  \om) \de B(s)\,. 
\end{eqnarray}
Assume that $F$ and $\frac{\partial }{\partial x} F(t, x, \om)$ 
are  H\"older continuous in $t$  of exponent greater than $1-H$ and Lipschitz in    $x$, $\HH$-differentiable in $\om$. Let $\xi:\Om\rightarrow \RR$ be measurable.   Let $h$ and $\Ga$ be defined as in Lemma \ref{t.3.1}  and assume 
that $\Ga:[0, T]\times \Om\rightarrow \Om$ is continuously differentiable in $s$ 
with respect to the topology  of $\HH$ (namely, $\frac{d}{ds} \Ga(s)\in \HH$).  
Then 
\begin{eqnarray}
&&F(t, G(t), \Ga(t))
= \eta(\xi(\om), \om) +\int_0^t f_0(s, G(s), \Ga(s)) ds \nonumber \\
&&\quad \qquad+\int_0^t 
f_1(s, G(s), \Ga(s)) \de B(s) +\int_0^t f_1(s, G(s), \Ga(s)) h(s, s, \om) ds \nonumber \\
&&\quad \qquad +\int_0^t \frac{\partial }{\partial x}F(s, G(s), \Ga(s)) g_0(s) ds
+\int_0^t \frac{\partial }{\partial x}F(s,  G(s), \Ga(s)) g_1(s) \de B(s)
 \nonumber \\
&&\quad \qquad +\int_0^t   (\DD F)(s, G(s),  \Ga(s)) \frac{d}{ds}   \Ga(s) \,  ds  \,, \label{e.ito-formula} 
\end{eqnarray}
where and in what follows we denote 
\[
(\DD F)(s, G(s),  \Ga(s)) \frac{d}{ds}   \Ga(s)
:=\DD F(s, x, \om) \bigg|_{x=G(s)\!,\, \om =\Ga(s)}  \frac{d}{ds}   \Ga(s)\,.
\]
\end{theorem}
\begin{proof}
Let $\pi:0=t_0<t_1<\cdots<t_n=t$ be a  partition of $[0, t]$ and denote $|\pi|=\max_{0\le k\le n-1}(t_{k+1}-t_k)$.   We have
\begin{eqnarray*}
&&F(t, G(t), \Ga(t))-F(0, G(0), \Ga(0))\\
& &\qquad =   \sum_{k=0}^{n-1} \left[ F(t_{k+1}, G(t_{k+1}), 
\Ga(t_{k+1}))- F(t_{k }, G(t_{k }), \Ga(t_{k }))\right]\\
& &\qquad = I_1+I_2+I_3\,,
\end{eqnarray*}
where
\begin{eqnarray*} 
I_1&=& \sum_{k=0}^{n-1} \left[ F(t_{k+1}, G(t_{k+1}), 
\Ga(t_{k+1}))-F(t_{k }, G(t_{k+1}), 
\Ga(t_{k+1}))\right] \,;\\ 
I_2&=& \sum_{k=0}^{n-1} \left[ F(t_{k}, G(t_{k+1}), 
\Ga(t_{k+1}))-F(t_{k }, G(t_{k }), 
\Ga(t_{k+1}))\right]\,; \\
I_3&=& \sum_{k=0}^{n-1} \left[ F(t_{k }, G(t_{k }), 
\Ga(t_{k+1}))-F(t_{k }, G(t_{k }), 
\Ga(t_{k }))\right]\,.
\end{eqnarray*} 
Let us first look at $I_1$.    Using Lemma \ref{t.3.1},   we have 
\begin{eqnarray*}
I_1
&=& \sum_{k=0}^{n-1} \int_{t_k}^{t_{k+1}} f_0(s, G(t_{k+1}), \Ga(t_{k+1})) ds\\
&&\qquad 
+\sum_{k=0}^{n-1} \int_{t_k}^{t_{k+1}} f_1(s,  x, \om)\de B(s)\Big|_{
x=G(t_{k+1}), \om=\Ga(t_{k+1}) }\\
&=&  \sum_{k=0}^{n-1} \int_{t_k}^{t_{k+1}} f_0(s, G(t_{k+1}), \Ga(t_{k+1})) ds
+\sum_{k=0}^{n-1} \int_{t_k}^{t_{k+1}} f_1(s,  G(t_{k+1}), \Ga(t_{k+1})) \de B(s) \\
&&\qquad + \sum_{k=0}^{n-1} \int_{t_k}^{t_{k+1}} f_1(s,  G(t_{k+1}), \Ga(t_{k+1})) h(t_{k+1}, s) ds\,. 
\end{eqnarray*}
From here it is easy to see that  
\begin{eqnarray}
\lim_{|\pi|\rightarrow 0} I_1
&=&   \int_0^t f_0(s, G(s), \Ga(s)) ds +\int_0^t 
f_1(s, G(s), \Ga(s)) \de B(s)\nonumber \\
&&\qquad +\int_0^t f_1(s, G(s), \Ga(s)) h(s, s, \om) ds \,. 
\label{proof-t.3.2-i1}
\end{eqnarray} 
Using the mean value theorem we have, denoting $G_{k, \theta}= 
G(t_k)+\theta \left[G(t_{k+1})
-G(t_k)\right]$, 
\begin{eqnarray*}
I_2
&=&  \sum_{k=0}^{n-1} \int_0^1 \frac{\partial }{\partial x}F(t_{k}, G_{k, \theta}  , 
\Ga(t_{k+1})) d\theta \left[ G(t_{k+1})-G(t_k)\right]\\
&=& \int_0^1 d\theta \left\{ \sum_{k=0}^{n-1} \int_{t_k}^{t_{k+1}} \frac{\partial }{\partial x}F(t_{k}, G_{k, \theta}  , 
\Ga(t_{k+1}))  \left[  g_0(s) ds   +g_1(s)\de B(s)\right] \right\}\,. 
\end{eqnarray*}
Since for any $\theta \in [0, 1]$,    $G_{k, \theta}$  is any point between $G(t_k)$ and $G(t_{k+1})$,
we see that for any $\theta\in [0, 1]$,  
\begin{eqnarray}
&&\lim_{|\pi|\rightarrow 0}  \sum_{k=0}^{n-1} \int_{t_k}^{t_{k+1}} \frac{\partial }{\partial x}F(t_{k}, G_{k, \theta}  , 
\Ga(t_{k+1}))  \left[  g_0(s) ds   +g_1(s)\de B(s)\right]  \nonumber\\
&&\qquad  
=\int_0^t \frac{\partial }{\partial x}F(s, G(s), \Ga(s)) g_0(s) ds
+\int_0^t \frac{\partial }{\partial x}F(s,  G(s), \Ga(s)) g_1(s) \de B(s)\,. \nonumber
\end{eqnarray} 
This implies 
\begin{eqnarray}
\lim_{|\pi|\rightarrow 0} I_2
=\int_0^t \frac{\partial }{\partial x}F(s, G(s), \Ga(s)) g_0(s) ds
+\int_0^t \frac{\partial }{\partial x}F(s,  G(s), \Ga(s)) g_1(s) \de B(s)\,.  \nonumber\\
\label{proof-t.3.2-i2} 
\end{eqnarray} 
 $I_3$ can be computed as follows.
\begin{eqnarray*}
I_3
&=& 
\sum_{k=0}^{n-1} \int_{t_k}^{t_{k+1}}   \DD F(t_k, G(t_k), \Ga(s)) \frac{d}{ds} \Ga(s) 
 \,    ds \,.
\end{eqnarray*}
From here it follow easily  
\begin{eqnarray}
\lim_{|\pi|\rightarrow 0} I_3
&=&   \int_0^t    \DD  F(s, G(s), \Ga(s)) \frac{d }{d  s}\Ga(s)    ds \,. 
\label{proof-t.3.2-i3}
\end{eqnarray}  
Now we combine \eref{proof-t.3.2-i1},\eref{proof-t.3.2-i2} and \eref{proof-t.3.2-i3}
to prove the  theorem. 
\end{proof}

\begin{remark}
 \begin{enumerate}
\item[(i)] $F(0, G(0), \Ga(0))=\eta(\xi(\om), \om)$. Thus, we can replace 
$ \eta(\xi(\om), \om)$ in \eqref{e.ito-formula}    by 
$F(0, G(0), \Ga(0)) $. 
\item[(ii)] We may also write the above  It\^o     formula 
\eqref{e.ito-formula}   in differential form as follows.
\begin{eqnarray}
&&dF(t, G(t), \Ga(t))
= f_0(t, G(t), \Ga(t))dt +f_1(t, G(t), \Ga(t))\de B(t)\nonumber\\ 
&&\qquad  +  \frac{\partial }{\partial x}F(t, G(t), \Ga(t)) \left[ g_0(t) dt
 + g_1(t) \de B(t)\right] 
 \nonumber \\
&&\qquad   + f_1(t, G(t), \Ga(t)) h(t, t, \om) 
dt +   (\DD  F)(t, G(t),  \Ga(t))  \frac{\partial }{\partial t} \Ga(t)  dt\,. 
\nonumber\\
\end{eqnarray}
\end{enumerate} 
\end{remark}
If $F$ and $ G$  are given by It\^o      integrals, then what will be the It\^o     formula?    Similar to the argument  in the  proof of  Proposition \ref{p.2.2} we can use the relationship  
\eqref{e.ito-pathwise-relation}    to 
obtain an analogous It\^o     formula for It\^o     integrals.  

Let 
\begin{eqnarray}
F(t, x, \om)
&=& \eta(x, \om)+\int_0^t f_0(s,x,\om)ds+\int_0^t f_1(s, x, \om) dB(s)\\
G(t,  \om)
&=& \xi( \om)+\int_0^t g_0(s, \om)ds+\int_0^t g_1(s,  \om) dB(s)\,.
\end{eqnarray}
  Let $h$ and $\Ga$ be defined as in Lemma \ref{t.3.1}.  We assume the conditions in Theorem \ref{t.3.2} hold. Moreover, we also assume that $\DD_sf_1^\phi(s)$ and $\DD_sg_1^\phi(s) $ are continuously 
  in $s$. From 
\eqref{e.ito-pathwise-relation} we have 
(we omit the explicit dependence on $\om$)
\begin{eqnarray*}
F(t, x )
&=& \eta(x )+\int_0^t \left[ f_0(s,x) -\DD_s^\phi  f_1(s,x)\right] ds+\int_0^t f_1(s, x ) \de B(s)\\
G(t )
&=& \xi +\int_0^t \left[ g_0(s)-\DD_s^\phi g_1(s)\right] ds+\int_0^t g_1(s )\de
B(s)\,.
\end{eqnarray*}
By the It\^o     formula  \eqref{e.ito-formula}     we have
\begin{eqnarray*}
&&F(t, G(t), \Ga(t))
= \eta(\xi)+\int_0^t \left[ f_0-\DD_s^\phi f_1\right]
(s, G(s), \Ga(s)) ds \\
&&\qquad  + \int_0^t f_1(s, G(s), \Ga(s))\de B(s)+\int_0^t f_1(s, G(s), \Ga(s))
h(s,s, \om ) ds\\
&&\qquad +\int_0^t \frac{\partial }{\partial x} F(s, G(s), \Ga(s))\left[
g_0 -(\DD_s^\phi g_1) \right] (s, \Ga(s)) ds\\
&&\qquad + 
\int_0^t \frac{\partial }{\partial x} F(s, G(s), \Ga(s))\ g_1  (s, \Ga(s)) \de B(s) \nonumber\\
&&\qquad 
+\int_0^t (\DD F)(s, G(s), \Ga(s)) \frac{d}{ds}\Ga(s) ds\,. 
\end{eqnarray*}
Using again the relationship \eqref{e.ito-pathwise-relation}  between pathwise   and It\^o     integral, we can rewrite the above identity as
\begin{eqnarray}
&&F(t, G(t), \Ga(t))\nonumber\\
&&= \eta(\xi)+\int_0^t \left\{  \DD_s^\phi
[f_1(s, G(s), \Ga(s))]+\left[ f_0   -\DD_s^\phi f_1\right]
(s, G(s), \Ga(s)) \right\} ds \nonumber\\
&&  + \int_0^t f_1(s, G(s), \Ga(s)) d B(s)+\int_0^t f_1(s, G(s), \Ga(s))
h(s,s, \om ) ds \nonumber\\
&&\  +\int_0^t \frac{\partial }{\partial x} F(s, G(s), \Ga(s))\left[
g_0 -(\DD_s^\phi g_1) \right] (s, \Ga(s)) ds \nonumber\\
&& + 
\int_0^t \DD_s^\phi \left[ \frac{\partial }{\partial x} F(s, G(s), \Ga(s))\ g_1  (s, \Ga(s))\right]  ds\nonumber\\
&& + 
\int_0^t \frac{\partial }{\partial x} F(s, G(s), \Ga(s))\ g_1  (s, \Ga(s)) d B(s) \nonumber\\
&& 
+\int_0^t (\DD F)(s, G(s), \Ga(s)) \frac{d}{ds}\Ga(s) ds\,. 
\end{eqnarray}

 \section{An iteration  principle}
 After transforming the original equation \eqref{e.ito-equation} 
 into the system of equations  \eqref{e.1.6} (or   
 \eqref{e.5.6} in next section for 
  general random coefficient case),   we can now use the Picard iteration method to solve 
 the new   differential system (pathwise).  But  the second equation in \eqref{e.1.6} involves 
 $\int_0^t \si(s, z(s), \Ga(s))\de B(s)$.  Since $B(s)$ is not differentiable, one cannot no longer 
 use the powerful Gr\"onwall  lemma.  One way to get  around this difficulty is to use the Besov spaces (see e.g. \cite{nualartrascanu}).   Here,  we propose to use the H\"older spaces which seems to be simpler. The idea is motivated by the works \cite{hunualartabel,  hunualarttams}. 
 
 In this section  we present a general contraction principle, which may   be useful 
 in solving other equations driven  by  H\"older continuous 
 functions. We call this approach the fractional Picard iteration.
 In next section we shall use this general contraction principle to solve \eqref{e.1.6} and subsequently to solve \eqref{e.ito-equation}.

Let $\BB$ be a separable Banach space with norm $\|\cdot\|$ (in case we need 
to specify we write $\|\cdot\|_\BB$).  We denote  by 
$\BB[0, T]$ the Banach space  of all continuous functions from $[0, T]$ to $\BB$ 
with the sup norm $\|x\|_{0, T}=\sup_{0\le t\le T}\|x(t)\|_\BB$.
For any   $0\le a<b\le T$ and an element $x\in \BB[0, T]$,  we define
\[
\|x\|_{a, b, \be} =\sup_{a\le s<t\le b} \frac{\|x(t)-x(s)\|}{|t-s|^\be} 
\]
if the above  right hand side is finite.   We also use the notation
$\|x\|_{a, b}=\sup_{a\le t\le b}\|x(t)\| $ and in the  case when $a=b$, we denote 
$\|x\|_{a, a}= \|x(a)\| $.  

 Given $\De\in (0, T]$ and   $\be\in (0, 1]$ we denote 
 \[
\||x\||_{\De,   \be} =\sup_{0\le t\le T}|x(t)|+\sup_{0\le s<t\le T, t-s\le \De}
\frac{\|x(t)-x(s)\|}{|t-s|^\be} \,. 
\]
Denote 
\[
\BB^{\De, \be}[0, T]=\left\{ x\in \BB[0, T]\,; \ \||x\||_{\De,   \be}<\infty\right\}\,. 
\]
As in Theorem 1.3.3 of \cite{kufner}, it is easy to verify that 
$\||x\||_{\De,   \be}$ is a  norm  and $\BB^{\De, \be}[0, T]$ is a Banach space
with respect to this norm.    When we need to emphasize 
the interval we may also add the  interval  into the notation, namely,
we may write $\||x\||_{a, b, \De, \be}$.  If $\De$ is clear, we omit the dependence on
$\De$ and write $\BB^{  \be}[0, T]= \BB^{\De, \be}[0, T]$.

We shall consider a mapping $F$ from   $ \BB[0, T]$ 
into itself.  Thus, for any element $x\in   \BB[0, T]$, $F(x)$ is a 
function from $[0, T]$ to $\BB$.  We can thus write such function as
$F(t, x)$.   We say $F$ is progressive if for any $a\in [0, T]$,
$\left\{F(t, x)\,,\ 0\le t\le a\right\}$  depends  only on $\{x(t), 0\le t\le a\}$.
In other words,  if  $x(t)=y(t)$  for all $t\in [0, a]$,  then   
$F(t,x)=F(t, y)$ for all  $t\in [0, a]$.  
 
 Here is the main theorem of this section.  
 \begin{theorem}\label{t.fix-point-theorem} 
 Let $\BB$ be a separable Banach space  and let    $F $   be a progressive 
 mapping from
 $ \BB[0, T]  \rightarrow 
 \BB [0, T]\ $ such that  $F(0, x)\in \BB$ is independent of 
 $x$ (it is equivalent to say that
 $F(0, x)\in \BB $ is independent of 
 $x(0)$).    Suppose that 
 there are  constants $\kappa$, $\Delta>0$, $\gamma$,  $\be \in (0, 1]$  and there 
 is a positive function $h:\RR^4\rightarrow \RR$,  increasing in all of its arguments,   such that  the following statements are true. 
 \begin{enumerate}
 \item[(i)]  For any   $0\le a<b\le T$ with $b-a\le \De$ and for any $x \in \BB ^{\De, \be }[0, T]$ 
 we have  
 \begin{equation}
 \|F ( x )\|_{a, b, \be }
 \le   
\kappa \big(1+\|x \|_{0, a} +
 \|x \|_{a,b,\be }   (b-a)^\ga  \big)\,. \label{e.F-holder-bound}
\end{equation}
 \item[(ii)] For any  $0\le a<b\le T$ with $b-a\le \De$ and   for any $x_1,
 x_2 \in \BB ^{\De, \be }[0, T]$ 
 we have  
\begin{eqnarray}
 \|F ( x_1     )-
F ( x_2)\|_{a, b, \be } 
& \le &  \bar h_{a, b}(x_1, x_2)
\bigg\{  \|x_1-x_2\|_{0, a} \nonumber\\
&&\quad + \|x_1 -x_2\|_{a,b,\be }   (b-a)^\ga \bigg\} \,, 
\label{e.F-diff-holder-bound}
\end{eqnarray}
where
\begin{eqnarray}
\bar h_{a, b}(x_1, x_2)
 =  h\big(\|x_1 \|_{0, a}, \|x_2 \|_{0, a},    \|x_1 \|_{a, b, \be },\|x_2 \|_{a, b, \be }\big)\,.  
\label{e.def-barh}  
\end{eqnarray}
 \end{enumerate}
Then   the mapping $F :\BB^{\be }[0, T]\ 
\rightarrow \BB^{\be }[0, T] $ has a unique 
fixed point $x\in \BB[0, T]$.  This means 
that there is a unique $x\in \BB[0, T]$ such that 
$x(t)=F(t, x )$ for all $t\in [0, T]$.   Moreover, 
there is a $\tau_0>0$ such that 
\begin{equation}
\begin{cases}  
      \|x \|_{0, T  } 
 \le  c_2 e^{c_1 \kappa^{1/\ga}T } (1+\|F(0)\|)  \,,   \\
 \sup_{0\le a< b\le T, b-a\le \tau_0} \|x \|_{a, b, \be } 
\le    c_2 e^{c_1 \kappa^{1/\ga}T }  (1+\|F(0)\|) \,, 
\end{cases} \label{e.bound-holder-x-theorem}
\end{equation} 
where $c_1$ and $c_2$ are two constants depending only on $\De$. 
\end{theorem} 
 
 \begin{proof}   We divide the proof into several steps.
  
 \noindent {\sl Step 1}. \  First, we prove that there is a  $\tau_1\in (0, \Delta]$   (the choice of 
 $\tau_1$ will be made more precise   later)  such that  $F$  has a unique fixed point 
on the interval $[0, \tau_1]$.    To this end we use Picard iteration.  We define   $ x_0 (t)  = F(0, x) $ for all $t\in [0, \tau_1]$
which is an element in $\BB$ by our assumption that $F(0, x)$ is independent of $x$.   We also define for $n=0, 1, 2, \cdots$ 
\begin{equation}
x_{n+1}(t) =F (t, x_ { n } )\,, \quad 
     t\in  [0, \tau_1]\,. 
\end{equation}

It is easy to see by the assumption  that $x_n(0)=F(0,x )$ for all $n\ge 0$.
From the assumption \eqref{e.F-holder-bound}, we have
\begin{eqnarray*}
\|x_{n+1}\|_{0, \tau_1, \be}
&\le&  \kappa (1+\|x_n(0)\|+\|x_n\|_{0, \tau_1, \be} \tau_1^\ga)\\
&\le&  \kappa (1+\|F(0)\|+\|x_n\|_{0, \tau_1, \be} \tau_1^\ga)\,.
\end{eqnarray*} 
Let 
\begin{equation}
\tau_1\le \frac{1}{(2\kappa)^{1/\ga}}\wedge \De \,. 
\label{e.cond1-tau}
\end{equation}\
[One can take $\tau_1\le \frac{1}{(2\kappa)^{1/\ga}}\wedge \De $.]
Then we have
\begin{equation}
\|x_{n+1}\|_{0, \tau_1, \be } 
 \le    \kappa (1+ \|F(0, x)\|)+\frac12 \|x_n\|_{0, \tau_1, \be}  \,. 
 \label{e.proof-x(n+1)by-xn}
\end{equation} 
By induction, we have
\begin{equation}
\sup _{n\ge 0} \|x_n\|_{0, \tau_1, \be}\le 2\kappa (1+ \|F(0)\|)\,. 
\end{equation}
Now by the fact that $\|x_n\|_{0, \tau_1}\le \|x_n(0)\|+\|x_n \|_{0, \tau_1, \be} \tau_1^\ga$, we see that
\begin{equation}
\sup _{n\ge 0} \|x_n\|_{0, \tau_1 }\le    \|F(0)\|
+ 2\kappa  (1+\|F(0)\|) \tau_1^ \ga 
\le 2 (1+\|F(0)\|) \,. 
\end{equation}
By the definition   \eqref{e.def-barh} of $\bar h$ we see that
\[
\sup _n \bar h_{0, \tau_1} (x_{n-1}, x_n)\le M_1<\infty 
\]
for some positive constant $M_1\in (0, \infty)$. Notice that 
$  x_n(0)=x_{n-1}(0)$.   Thus condition
\eqref{e.F-diff-holder-bound}   gives 
\[
\|F(x_{n+1})-F(x_n)\|_{0, \tau_1, \be}
\le M_1 \|x_{n+1}-x_n\|_{0, \tau_1, \be} \tau_1^\ga \,.
\]
Choose 
\begin{equation}
\tau_1\le \frac{1}{(2M_1)^{1/\ga}}\wedge \frac{1}{(2\kappa)^{1/\ga}}\wedge \De\,.
\label{e.cond2-tau} 
\end{equation}
Then we have
\begin{eqnarray*}
\| x_{n+1} - x_n  \|_{0, \tau_1, \be}
&=& \|F(x_{n })-F(x_{n-1})\|_{0, \tau_1, \be} \\
&\le&  \frac12 \left \| x_{n } - x_{n-1} \right \|_{0, \tau_1, \be} \,. 
\end{eqnarray*}
Since $x_n(0)=F(0)$ for all $n$ this means that 
   $\{x_n\}$ is a Cauchy sequence in $\BB^\be [0, \tau_1]$ and  
it converges to an element  $x\in  \BB^\be [0, \tau_1]$.  Obviously, this limit $x$ 
is the unique solution to   $x(t) =F(t, x)$ for $t\in [0, \tau_1]$.  
Clearly, the limit satisfies
\begin{equation} 
\sup _{n\ge 0} \|x \|_{0, \tau_1 }    
\le 2 (1+ \|F(0)\|) \, ,\quad 
  \|x \|_{0, \tau_1, \be  }    
\le 2 \kappa (1+ \|F(0)\|)  \, . 
\end{equation} 
  
\noindent{\sl Step 2}. \  Now we  explain the inductive argument  to construct 
a unique solution on the interval   $  [0, T\wedge 
T_{k+1} ]$  from a solution
on    $ [0, T\wedge 
T_k ]$,  where $T_k=\tau_1+\cdots+\tau_k$.
For any positive integer $k\ge 1$ assume that there
is a unique solution $x(t), t\in [0, T_k]$ satisfying $x(t)=F(t, x), t\in [0, T_k]$.
We want to construct a  unique solution $x(t), t\in [0, T_{k+1}]$ satisfying $x(t)=F(t, x), t\in [0, T_{k+1}]$ for some $\tau_{k+1}>0$ 
(see below for the definition of $\tau_{k+1}$).  To simplify notation we assume $ 
T_{k+1}\le T$ (or we replace $T_{k+1}$ by $T_{k+1}\wedge T$). 
Define the following sequence (still use $x_n$) 
\begin{equation}
\begin{cases}
&   x_0(t)=
\begin{cases}
x(t)\,, &\qquad\qquad\quad   \hbox{when $0\le t\le T_k$}\,, \\
x(T_k)\,, &\qquad\qquad\quad   \hbox{when $T_k \le t\le T_{k+1}$} \,, 
\end{cases} \nonumber\\
&  x_{n+1}(t)=F(t, x_n)\,, \qquad\qquad  \hbox{for all $0\le t\le T_{k+1}$} 
 \,, \label{e.zn-sup}
\end{cases}
\end{equation} 
where $n=0, 1, \cdots$.  
  Since $x(t), t\in [0, T_k]$ is the unique solution to $x(t)=F(t, x), t\in [0, T_k]$,
we see that $x_n(t)=x(t)$ for all $t\in [0, T_k]$.  
With exactly the same argument as for 
\eqref{e.proof-x(n+1)by-xn}, we have 
for any positive integer $k\ge 1$,  
\begin{eqnarray}
\|x_{n+1}\|_{T_k, T_{k+1}, \be } 
 &\le&    \kappa  (1+ \|x\|_{0, T_k} )
  +\frac12 \|x_n\|_{ T_k, T_{k+1}, \be }   
\end{eqnarray} 
under the condition \
\begin{equation}
\tau_{k+1}\le \frac{1}{(2\kappa)^{1/\ga}}\wedge \De\,.
\end{equation}
[We can take $\tau_{k+1}= \frac{1}{(2\kappa)^{1/\ga}}\wedge \De$]. 
This can be used (by  induction on $n$)  to prove 
\begin{equation}
\sup_{n\ge 0}  \|x_n\|_{ T_k, T_{k+1}, \be }
\le 2 \kappa  (1+ \|x\|_{0, T_k })=:M_{k+1}^{(1)}\,. 
\label{e.proof.bound-holder-xn}
\end{equation}
As a consequence, we have
\begin{eqnarray}
\sup_{n\ge 0}  \|x_n\|_{0, T_{k+1} }
&\le& \|x\|_{0, T_k} +  \|x_n\|_{ T_k, T_{k+1}, \be } \tau_{k+1}^\ga 
\nonumber\\
&\le&  (2 \kappa \tau_{k+1} ^\be+1)  (1+\|x\|_{0, T_k })
\nonumber\\
&\le&  2  (1+\|x\|_{0, T_k })=:M_{k+1}^{(2)}
\label{e.x-uniform-bound}
\,. 
\end{eqnarray}
Now letting  $M_{k+1}:=h(M_{k+1}^{(2)}, M_{k+1}^{(2)}, M_{k+1}^{(1)}, M_{k+1}^{(1)}   )$,  we have by \eqref{e.F-diff-holder-bound} 
\begin{eqnarray}
\|x_{n+1}-x_n\|_{ T_k, T_{k+1}, \be} 
&=& \|F(x_{n })-F(x_{n-1})\|_{ T_k, T_{k+1}, \be} \nonumber\\
&\le& M_{k+1} \|x_{n }-x_{n-1}\|_{ T_k, T_{k+1}, \be}  \tau_{k+1}^\ga \nonumber\\
&\le& \frac12  \|x_{n }-x_{n-1}\|_{ T_k, T_{k+1}, \be}  
\end{eqnarray}
if 
\begin{equation}
\tau_{k+1}\le\frac{1}{M_{k+1}^{1/\ga}}\wedge \frac{1}{(2\kappa)^{1/\ga }}\wedge \De\,.
\label{e.cond2-tau=k+1}
\end{equation}
Thus, under the above condition \ref{e.cond2-tau=k+1}, 
$\{x_n\}$ is a Cauchy sequence in $\BB  [0, T_{k+1}]$.
It has a unique limit $x$ which satisfies $x(t)=F(t, x)$ for all $t\in
[0,   T\wedge T_{k+1}]$.  Indeed,  the fact $x(t)$ satisfies $x(t)=F(t, x)$ for all $t\in
[0,   T\wedge T_{k}]$  follows from the inductive assumption. On $[T\wedge T_k, T\wedge T_{k+1}]$,  $x_n$ is a Cauchy sequence in $B^\be [T\wedge T_k, T\wedge T_{k+1}]$   and $F$ is continuous on  $B^\be [T\wedge T_k, T\wedge T_{k+1}]$ by the assumption 
\eqref{e.F-diff-holder-bound}.    It is also easy to verify  from \eqref{e.proof.bound-holder-xn}
and \eqref{e.x-uniform-bound} 
that the solution satisfies
\begin{equation}
\begin{cases}
\|x \|_{ T_k, T_{k+1}, \be }
\le 2 \kappa    (1+\|x \|_{0, T_k})\,;  \\ 
  \|x \|_{0, T_{k+1} }  
 \le   2  (1+ \|x \|_{0, T_k})\,.
 \end{cases} 
  \label{e.4.100} 
\end{equation}

\noindent {\sl Step 3}. 
 Denote  $T_\infty=T\wedge(\tau_1+\tau_2+\cdots)$.  By induction argument, we can construct a unique solution $x(t)$ on $t\in [0, T_\infty]$  for  the equation $x(t)=F(t, x)$.  We want to show 
 $T_\infty=T$.   To do this, the idea is to show that $\tau_k\ge \tilde \tau_0$   for some $\tilde \tau_0>0$ and for all $k\ge 1$.

From the same argument as for \eqref{e.proof.bound-holder-xn}
and \eqref{e.x-uniform-bound}  
we have
\begin{equation}
\begin{cases}
 & \|x \|_{ k\tau, (k+1)\tau, \be }
\le 2 \kappa  (1+ \|x \|_{0, k\tau})\,;  \\
&
  \|x \|_{0, (k+1)\tau }  
 \le   2  (1+ \|x \|_{0, k\tau })  
\end{cases}\label{e.proof-x-bound} 
\end{equation}
as long as  $\displaystyle \tau\le \frac{1}{(2\kappa)^{1/\ga}}$  and
$(k+1)\tau\le T_\infty$.  In fact, to obtain the above bounds, we only need to use the condition \eref{e.F-holder-bound}.

We choose $\displaystyle \tau= \frac{1}{(2\kappa)^{1/\ga}}\wedge \Delta $  and divide 
the interval $[0, T_\infty]$  into $N$ sub-intervals,  where 
\[
N=\left[\frac{T_\infty}{\tau}\right]+1=\left[\frac { T_\infty  } {\De}\right]
\vee \left[ T_\infty (2\kappa)^{1/\ga} \right]+1\,.
\] 
Denote $A_k= \|x \|_{0, k\tau }$.  The second inequality in \eqref{e.proof-x-bound}
can be written as 
\[
A_{k+1}\le 2 +2 A_k\,,  \quad k=0, 1, 2, \cdots
\]
 An
elementary induction argument yields
\begin{eqnarray*}
A_k 
&\le& 2+2^2+\cdots +2^k+2^k A_0\\
&\le& 2^{k+1} +2^{k }  A_0\,. 
\end{eqnarray*}
Thus, 
we see that 
\begin{eqnarray}
   \|x \|_{0, T_\infty  } 
&\le&   2^{N+1}+2^N  \|F(0)\| 
\nonumber\\
&\le&  c_2 e^{c_1 \kappa^{1/\ga}T } (1+\|F(0)\|)\  
\label{e.bound-uniform-x-proof}
\end{eqnarray}
for some constants $c_1$ and $c_2$ dependent only on   $\De$. 
This together with the first  inequality in \eqref{e.proof-x-bound}  yields
\begin{equation}
\|x \|_{ k\tau, (k+1)\tau, \be } 
\le  c_2 e^{c_1 \kappa^{1/\ga}T } (1+\|F(0)\|)
\label{e.bound-holder-x-proof}
\end{equation}
for any $k$ such that $(k+1)\tau\le T_\infty$ and for
$\displaystyle \tau= \frac{1}{(2\kappa)^{1/\ga}}\wedge \De$.

\noindent{\sl Step 4}. \ 
Denote  
\begin{equation*}
\begin{cases}
\tilde M_1&=\tilde M_2=  c_2 e^{c_1 \kappa^{1/\ga}T } (1+\|F(0)\|) \,;\nonumber\\
 \tilde M
 &= h(\tilde M_2, \tilde M_2, \tilde M_1, \tilde M_1)\,. \nonumber 
\end{cases}
\end{equation*}
Then from \eqref{e.proof.bound-holder-xn} and \eqref{e.x-uniform-bound}
we see that $M_{k+1}^{(1)}\le \tilde M_1$ and 
 $M_{k+1}^{(2)}\le \tilde M_2$ for all $k$.   Since $h$ is increasing in all of its arguments,
%
this means that 
 we can choose
  $\tau_k$ such that
  \[
  \tau_k\ge \tilde \tau:\displaystyle  = \frac{1}{\tilde M^{1/\ga}}\wedge 
 \frac{1}{(2\kappa)^{1/\ga}}\wedge \De\,, \quad \forall \ k\ge 1\,.
\]
  Since $\tilde \tau$ is independent of 
  $k$, we see that  $T_\infty=T$.
  
 The first  inequality \eqref{e.bound-holder-x-theorem} 
 is a straightforward consequence of \eqref{e.bound-uniform-x-proof}.
 With possibly a different choice of $c_1$ and $c_2$, we can write
 \eqref{e.bound-holder-x-proof}  as 
 \begin{equation*}
\|x \|_{ k\tau_0, (k+1)\tau_0, \be } 
\le    c_2 e^{c_1 \kappa^{1/\ga}T } (1+\|F(0)\|)\,. 
\end{equation*} 
If $a, b\in [k\tau_0,   (k+1)\tau_0]$,  then we see 
easily that the second inequality  in \eqref{e.bound-holder-x-theorem}
holds.  If $a\in   [(k-1)\tau _0,  k\tau_0]$
and $b\in   [k\tau_0,   (k+1)\tau_0]$,  then 
\begin{eqnarray*}
\frac{ \|x(b)-x(a)\|}{|b-a|^\be}
&\le& \frac{ \|x(b)-x(k\tau_0)\|+\|x(k\tau)-x(a)\|}{|b-a|^\be}\\
&\le& \frac{  \|x(b)-x(k\tau_0)\|}
{|b-k \tau_0|^\be}+\frac{\|x(k\tau)-x(a)\|}{|k\tau_0-a|^\be}\\
&\le&   \|x\|_{(k-1)\tau _0,  k\tau_0, \be} +\|x\|_{k\tau _0,  (k+1)\tau_0, \be}\\
&\le& 2 c_2 e^{c_1 \kappa^{1/\ga}T }(1+\|F(0)\|)\,.
\end{eqnarray*}
Up to a different choice of constant $c_2$, we prove 
the second inequality in \eqref{e.bound-holder-x-theorem}.
\end{proof}
 
\begin{remark} If $\bar h$ has some particular form, one may obtain
some stability results for the solutions. Namely, if $x_1$ and $x_2$ are two solutions with different initial conditions, or with different 
$F$, one may bound $\|x_2-x_1\|$ (see analogous results \cite{hunualartabel, hunualarttams}).  However,  we shall not persuade this problem.  
\end{remark}  
 
\section{General stochastic differential equations}
\subsection{Reduction of the equation}
Let $b\,, \si:[0, T]\times \RR\times \Om\rightarrow \RR$ be
a measurable mapping.
We shall specify
the conditions on them later. Let $\eta$ be a given random variable.
The main objective of this paper is to study  the following
 It\^o     stochastic differential equation
\begin{equation}
\begin{cases}
dx(t)=b(t, x(t), \om)dt+\si(t, x(t), \om)dB(t)\,, \quad 0\le t\le T\,, \\
x(0)=\eta\,,
\end{cases} \label{e.5.1} 
\end{equation}
where $dB(t)$ is the It\^o     differential.  

We can use the argument as in the introduction (see
e.g.  \eqref{e.1.6})   to reduce  the above equation
\eqref{e.5.1}, now with   random  coefficients.
Using the relationship  \eqref{e.ito-pathwise-relation}
 between the It\^o       and pathwise   stochastic
integrals   and the chain rule for   derivative we have
\begin{eqnarray*}
\int_0^t \si(s, x(s), \om)dB(s)
&=&\int_0^t \si(s, x(s), \om)\de B(s) -\int_0^t \DD^\phi_s \left[\si(s, x(s), \om)
\right]ds \\
&=& \int_0^t \si(s, x(s), \om)\de B(s) -\int_0^t \DD^\phi_s \left[\si\right]
(s, x(s), \om)ds \\
&&\qquad -\int_0^t \si_x (s, x(s), \om)\DD^\phi_s x(s) ds \,,
\end{eqnarray*}
where $\si_x$ denotes the partial derivative of $\si$ with respect to
$x$, and $\DD^\phi_s \left[\si\right]$ denotes the partial derivative of
$\si$ with respect to the random element $\om$.
Thus,    the equation \eqref{e.5.1}  may be written as
\begin{eqnarray}
x(t)&=&\eta+\int_0^t \tilde b(s, x(s), \om)ds+
\int_0^t \si(s, x(s), \om)\de B(s)\nonumber\\
&&\qquad 
 -\int_0^t \tilde \si (s, x(s), \om)\DD^\phi_s x(s) ds\,, 
 \label{e.5.5}
\end{eqnarray}
where 
\begin{equation}
\begin{cases}\tilde b(s, x, \om):=b(s, x, \om)- \DD^\phi_s \left[\si\right]
(s, x, \om)\\ 
\tilde \si(s, x, \om)=  \si_x(s, x, \om)\,. 
\end{cases} 
\end{equation} 
 
As explained in the introduction, this equation can be considered as a  
  first order nonlinear hyperbolic equation of infinitely many variables,
driven by fractional Brownian  motion,  where $\om \in \Om$ is considered 
as an infinite dimensional variable.  We shall use the elementary characteristic curve 
method.  The characteristic equation will be an equation in $\Om$ which takes the 
form of the first equation of the following system of equations. 
This means that 
to solve the above equation \eqref{e.5.5}  we  will first solve  the following
coupled system of equations (which we call  it the system 
of characteristic equations corresponding to
\eqref{e.5.1}).   
\begin{equation}
\begin{cases}\Gamma(t)=\om+\int_0^t  \tilde \si  (s, z(s), \Gamma(s))\int_0^\cdot \phi(s,
u)du ds\,;\\ \\
z(t)= \eta(\om ) +\int_0^t \tilde b(s, z(s), \Gamma(s))ds
+\int_0^t \si(s, z(s),
\Gamma(s))\de B(s)\\ 
\qquad\qquad  +\int_0^t\int_0^s  \si(s, z(s), \Ga(s)) \tilde \si (u, z(u),
\Ga(u)) \phi(s, u) duds\,. 
\end{cases} \label{e.5.6}
\end{equation}
We shall show that the solution to equation \eqref{e.5.6} 
can be used to express the solution of \eqref{e.5.5}. However, first  we need to show that 
\eqref{e.5.6} has a (unique) solution.
 
\subsection{Solution to the reduced equation}
In this section we prove that the system 
\eqref{e.5.6} has a unique solution.  
When the Hurst parameter $H>1/2$ and in the absence of $\DD^\phi_s x(s)$,  the equation
\eqref{e.5.5} has been studied by many authors (see \cite{nualartrascanu, zahle, hustochastics}   for a recent study and also for some more references).   We only mention two works.  
In \cite{nualartrascanu},   Besov spaces  are used to accommodate the solutions. 
In \cite{hunualartabel}, the solution is shown to
be H\"older continuous and the stability with respect to H\"older norm is also studied in that paper  (see \cite{hunualarttams}
for a similar study when the Hurst parameter $H\in (1/3, 1/2]$).
Here,   we shall use the H\"older spaces  together with the general contraction principle 
established in Section 4   to prove the existence and uniqueness of the solution.
Our  idea to solve the equation
\eqref{e.5.5} seems   also new even in the classical case    (namely, in the absence of $\DD^\phi_s x(s)$ in \eqref{e.5.5}).  

The system of (two)  equations \eqref{e.5.6} will be solved for any fixed 
$\om\in \Om$.  This means that we are going to find pathwise solution 
of  \eqref{e.5.6} by using  Theorem \ref{t.fix-point-theorem}.
To this end, we rewrite the equation \eqref{e.5.6} with a replacement of 
 $\Ga$  of by $\Ga+\om$ (we use the same notation $\Ga(t)$ without ambiguity).  
\begin{equation}
\begin{cases}\Gamma(t)= \int_0^t  \tilde \si (s, z(s), \Gamma(s)+\om)\int_0^\cdot \phi(s,
u)du ds\,;\\ \\ 
z(t)= \eta(\om ) +\int_0^t \tilde b(s, z(s), \Gamma(s)+\om)ds
+\int_0^t \si(s, z(s),
\Gamma(s)+\om)\de B(s)\\ 
\qquad\qquad  +\int_0^t\int_0^s  \si(s, z(s), \Ga(s)+\om) \tilde \si(u, z(u),
\Ga(u)+\om) \phi(s, u) duds\,. 
\end{cases} \label{e.to-solve}
\end{equation}

Before we solve \eqref{e.to-solve},  we need to explain the space  that the solution  stay. 
 To find such a  space to accommodate the above $\Ga$ we introduce the following
Banach space:
\begin{eqnarray}
\HH
&=&\HH_p=\bigg\{ h: \cT\rightarrow \RR\,; \ \hbox{$h$ is absolutely 
continuous}    \nonumber\\
&&\qquad\hbox{such that}\  \int_0^T |\dot h(s)|^p ds<\infty\bigg\}\,,
\end{eqnarray}
where $p$ is a  number such that $p\in (1, \frac{1}{2-2H})$ 
(we shall fix such a number throughout the remaing part of this paper),  and  where    the norm is defined by
\[
\|h\|_\HH =\|h\|_{\HH_p} =\left( \int_0^T | \dot h(s)|^p ds\right)^{1/p}\,. 
\]
It is straightforward to see that any $h\in \HH$ is an element
of $\Om$ and
\[
\|h\|_\Om\le c_{p, T} \|h\|_\HH\,.
\]
\begin{remark} 
The principle to choose the Banach space $\HH$ is as follows. 
First, we need that $\frac{d}{dt}\Ga(t)\in \HH$. 
Secondly, we want the norm of $\HH$ is as strong as possible
so that the coefficients $\sigma$, $\tilde \si$, and $\tilde b$ are differentiable
on $\om$ with respect to this norm $\|\cdot\|_\HH$. Namely, with respect to $\om$,   the coefficients
$\si$ and $b$ and the  initial condition $\eta$ satisfy 
\begin{equation}
\left| \si( \om+h)-\si(\om)\right|\le C \|h\|_\HH \,, \quad \forall \ h\in \HH\,. 
\label{e.5.8}
\end{equation}
(Similar inequality for $\tilde b$, $\tilde \si$,  and the initial conditions).  
Of course, the larger the norm of $\HH$,   the  broader  the condition  
\eqref{e.5.8} will be.  In the  analysis of nonlinear Wiener functionals, it is known that many interesting 
random variables do  not satisfy  \eqref{e.5.8} with $\HH=\Om$ 
(see the  example  of L\'evy area in \cite{hugaussian}).  But usually \eqref{e.5.8} is satisfied 
when $\HH$ is the Cameron-Martin norm, which is given by
$\|h\|_{\HH_\phi}^2:=\int_0^T\int_0^T \phi(u-v) h(u) h(v) dudv$, in our case of frcational Brownian motion.  Namely, for   many  
random variables in stochastic analysis, such as the solution of a stochastic differential equation, 
we have 
\begin{equation}
\left| f( \om+h)-f(\om)\right|\le C \|h\|_{\HH_\phi} \,, \quad \forall \ h\in \HH_\phi \,. 
\label{e.5.9}
\end{equation}
An inequality  of Littlewood-Paley type
(\cite{mmv}) states 
\[
\|h\|_{\HH_\phi }\le C_H \|h\|_{\HH_q}\,,  \quad \forall \ h\in \HH_q\,,\quad \hbox{with $q:=1/H$}\,.  
\]
When $H>2/3$,  we have $q=\frac{1}{H}< \frac{1}{2-2H}$.  In this case, we see that the 
\eqref{e.5.9} implies \eqref{e.5.8}  when we choose $p$ close to $\frac{1}{2-2H}$.   This means that the condition \eqref{e.5.8} is  satisfied
for many    random variables  we encounter when $H>2/3$.   
\end{remark}

Before we proceed to solve \eqref{e.to-solve},   we 
state  some assumptions on the coefficients $b$ and $\si$.  
\begin{hypothesis}\label{h.b-sigma} Let $\cL$ be a positive constant. 
The measurable functions $b, \si:\cT\times \RR\times \Om\rightarrow 
\RR$ satisfy the following conditions.  
\begin{enumerate}
\item[(i)] $b$ is continuously differentiable in $x$ and satisfies 
\[
\begin{cases}
  |b(t,x,\om)|
 \le   \cL(1+|x|)\,;\\
 \left|\frac{\partial }{\partial x} b(t,x,\om)\right|
 \le   \cL \,;\\
 \left\|\DD b(t,x,\om)\right\|_\HH 
 \le   \cL \,.   
\end{cases}
\]
\item[(ii)] $\si(t,x,\om)$ is twice continuously differentiable in $x$ with bounded first and second derivative and satisfies 
\[
\begin{cases}
 |\si(t,x,\om)|
 \le   \cL(1+|x|)\,; \\
  |\frac{\partial }{\partial t} \si(t,x,\om)-
  \frac{\partial }{\partial t} \si(t, y, \tilde \om)|
  \le \cL (|x-y|+\|\om-\tilde \om\|_\HH)\,;\\  
  |\frac{\partial}{\partial x} \si(t,x,\om)|+|\frac{\partial^2}{\partial x^2} \si(t,x,\om)|
 \le   \cL\,; \\ 
 |\DD_t^\phi \si(t,x,\om)|
 \le   \cL(1+|x|)\,;\\
 \|\DD    \si(t,x,\om) \|_\HH 
 +\|\DD ^2   \si(t,x,\om) \|_{\HH^2} \le \cL\,;\\ 
  \|\DD   \frac{\partial}{\partial x} \si(t,x,\om) \|_\HH \le \cL\,. 
\end{cases}
\]
%
\end{enumerate}
\end{hypothesis}

From now on we denote $\BB=\HH\oplus  \RR$ and we denote by
$\BB[0, T]$ the space of all continious functions from $[0, T]$
ro $\BB$.  Similarly, we will also use the notation $\HH[0, T]$.
$\RR[0, T]$ is then the space  $C([0, T])$ 
of all continuous functions from 
$[0, T]$ to $\RR$.  

Define a mapping from  $\BB[0, T]$ to $\BB[0, T]$ as follows.
\begin{equation}
\begin{cases}
F_1(t, \Ga, z):= \int_0^t  \tilde \si (s, z(s), \Gamma(s)+\om)\int_0^\cdot \phi(s,
u)du ds\,;\\ \\ 
F_2(t, \Ga, z):= \eta(\om ) +\int_0^t \tilde b(s, z(s), \Gamma(s)+\om)ds
+\int_0^t \si(s, z(s),
\Gamma(s)+\om)\de B(s)\\ 
\qquad\qquad\qquad   +\int_0^t\int_0^s  \si(s, z(s), \Ga(s)+\om) \tilde \si(u, z(u),
\Ga(u)+\om) \phi(s, u) duds\,. 
\end{cases} \label{e.def-F}
\end{equation}
 We also write $F_i(\Ga, z)=F_i(t, \Ga, z)$,  $i=1, 2$.  
 It is easy to see that for any $(\Ga, z)\in   \BB[0, T]$,
 $(F_1(\Ga, z), F_2(\Ga,z))$ is also in $\BB[0, T]$.  
\begin{lemma}\label{l.uniform-bound-f1} For any $\tau \in \cT$, if $(\Ga, z)\in \BB[0,\tau ]$,  then
$F_1(\Ga, z)\in \HH [0,  \tau ] $ and 
\begin{eqnarray}
\|\frac{d}{dt}F_1(\Gamma, z)\|_\HH&\le& \kappa  \,,  \label{e.5.17}\\
 \|F_1(\Ga, z)\|_{ 0, \tau }& \le&  \kappa \tau   \,,
\label{e.5.18}
\end{eqnarray} 
where and in what follows
$\kappa=c_{p, H, T, \cL}$ is a constant depending only on $p, H$,  $T$ 
and the  bound $\cL$ of the coefficients $b$ and $\si$, which may vary at different occurrences.   
\end{lemma}
\begin{proof} From the definition of $F_1$ we see 
\[
\frac{d}{dt} F_1(\Ga, z)=  \tilde \si (t, z(t), \Gamma(t)+\om)\int_0^\cdot \phi(t,
u)du \,.  
\]
Recall that we fix $p<\frac{1}{2-2H}$.   Since $\tilde \si$ is bounded and  $\phi(t,u)=H(2H-1) |t-u|^{2H-2}$, we have
\begin{eqnarray*}
\| \frac{d}{dt} F_1(\Ga, z)\|_\HH^p 
 \le  \kappa    \int_0^T
\phi(t,u)^p du
 =     \kappa \,. 
\end{eqnarray*}
This proves \eqref{e.5.17}.  
Similarly, since $F_1(0)=0$,  we have 
\begin{eqnarray*}
\| F_1(\Ga, z)\|_{ 0, \tau}\le \| \frac{d}{dt} F_1(\Ga, z)\|_{ 0, \tau }\tau
&=&    \kappa\tau  \,, 
\end{eqnarray*}
proving \eqref{e.5.18}. 
\end{proof}

\begin{lemma} \label{l.uniform-bound-f2}  
 For any $\tau\in \cT$, if $(\Ga, z)\in \BB[0, \tau]$,  then
$F_2(\Ga, z)\in C[0, \tau]$ and  for any $0\le a<b\le \tau$, 
\begin{eqnarray}
  \|F_2\|_{a,b, \be} 
 & \le&  \kappa    \left(1+ \|B\|_{a, b, \be}\right)
  \Big\{ 1+\|z\|_{0, a}\nonumber\\
  &&\qquad +
  \|z\|_{a, b, \be}(b-a)^{ \be}  +\|  \Ga\|_{a, b, \be } (b-a) ^\be   \Big\}\,. 
  \label{e.sup-f2} 
\end{eqnarray}
\end{lemma}
\begin{proof}  First,  we write 
\[
F_2=\eta(\om)+F_{21}+F_{22}+F_{23}\,,
\]
where
\[
\begin{cases}
 F_{21}(t)= \int_0^t \tilde b(s, z(s), \Gamma(s)+\om )ds\,;\nonumber\\
 F_{22}(t)= \int_0^t \si(s, z(s),
\Gamma(s)+\om)\de B(s)\,; \nonumber\\ 
 F_{23}(t) = \int_0^t\int_0^s  \si(s, z(s), \Ga(s)+\om) \tilde \si(u, z(u),
\Ga(u)+\om) \phi(s, u) duds\,. 
\end{cases}
\]
From the assumption on $b$ and $\DD_s^\phi \si$, we see that
for any $0\le a<b\le \tau$,  we have
\begin{eqnarray*}
|F_{21}(b) -F_{21}(a)|
&=& \int_a^b |\tilde b(s, z(s), \Gamma(s)+\om)|ds\\
&\le& \kappa    \int_a^b \left[1+|z(s)|\right] ds\\
 &\le& \kappa 
   \left[ 1+\|z\|_{ a, b }\right]  (b-a) \,. 
\end{eqnarray*}
This implies that
\begin{eqnarray}
\|F_{21}\|_{ a, b , \be}
&\le&   \kappa 
  \left[1+ \|z\|_{ a, b } \right] 
   \nonumber\\
&\le&   \kappa 
  \left[1+|z(a)|+ \|z\|_{ a, b, \be  } (b-a)^\be  \right] 
    \,. 
  \label{e.sup-f21} 
\end{eqnarray}
Now we consider $F_{23}$. We have
\begin{eqnarray*}
\left|F_{23}(b)-F_{23}(a)\right|
&=&\int_a^b\int_0^s  |\si(s, z(s), \Ga(s)+\om) \tilde \si(u, z(u),
\Ga(u)+\om) \phi(s, u)| duds  \\
&\le& \kappa  
\int_a^b\int_0^s \left[ 1+ |  z(s)|  \right] \phi(s, u)  duds  \\
&\le&  \kappa   \left[ 1+ \|z\|_{ a, b } \right] 
\int_a^b\int_0^s     \phi(s, u)  duds \\
&\le& \kappa   \left[ 1+ \|z\|_{ a, b } \right]   (b^{2H}-a^{2H})\\
&\le& \kappa   \left[ 1+ \|z\|_{ a, b } \right]  (b-a)\,. 
\end{eqnarray*}
This implies 
 \begin{equation}
\|F_{23}\|_{ a, b , \be}\le \kappa   \left[ 1+ \|z\|_{ 0, a } 
+ \|z\|_{ a, b } (b-a)^\be \right]     \,. 
 \label{e.sup-f23} 
 \end{equation} 
 $F_{22} $ is more complicated to handle because the fractional Brownian motion 
 $B$ is not differentiable. Denote
 $\si_r=\si(r, z(r), \Ga(r)+\om)$.  We have  for an $\al\in (1-\be, \be)$,  
 \begin{eqnarray*}
&&\left|\int_a^b \si(r, z(r), \Ga(r)+\om) \de B(r)\right|
= \left| \int_a^b D_{b-}^{1-\al} B_{b-}(r) D_{a+} ^\al \si_rdr \right| \\
&&\qquad \le  \kappa \|B\|_{a, b, \be} \Big|  \int_a^b (b-r)^{\al+\be-1} \Big\{
\frac{\si_r}{(r-a)^\al }+\int_a^r \frac{\si_r-\si_\rho}{(r-\rho)^{\al+1}}d\rho\Big\}dr\Big|\\
&&\qquad \le    \kappa   \|B\|_{a, b, \be}   \Big\{
(1+\|z\|_{a, b })  (b-a)^\be 
+(\|z\|_{a,b,\be} +\|\Ga\|_{a,b,\be}) (b-a)^{ 2\be}  \Big\} \,. 
 \end{eqnarray*}
This implies 
\begin{eqnarray}
  \|F_{22}\|_{a, b, \be}
  &\le & \kappa   \|B\|_{a, b, \be}   \Big\{
 1+\|z\|_{0, a}    + ( \|z\|_{a, b, \be}+\|\Ga\|_{a,b,\be}) 
 (b-a)^{ \be}\Big\}   \,. 
 \label{e.sup-f22} 
  \end{eqnarray} 
Combining \eqref{e.sup-f21}, \eqref{e.sup-f23} and \eqref{e.sup-f22}, we    prove the lemma.
 \end{proof}

To bound the H\"older norm of the difference, we first need  the following simple general result.
\begin{lemma}\label{l.one-comp-holder} Let $B_1$  and $B_2$ be 
two  Banach spaces
with  norms $\|\cdot\|_1$  and $\|\cdot\|_2$ and let $f: B_1\rightarrow B_2$ be twice continuously (Frechet) differentiable  
with bounded first and second derivatives.  
\begin{enumerate}
\item[(i)]  If $x_1, x_2, y_1, y_2\in B_1$,  then
\begin{eqnarray}
&&\|f(y_2)-f(y_1)-f(x_2)+f(x_1)\|_2
\le \|f' \|_\infty \| y_2-y_1- x_2+  x_1\|_1 \nonumber\\
&&\qquad \qquad +
\|f''\|_\infty \left[ \|y_1-x_1\|_1 +\|y_2-x_2\|_1 \right]\|
x_2-x_1\|_1   \,. 
\label{e.four-point}
\end{eqnarray}
\item[(ii)] 
Let $ x_1, x_2: [a, b]\rightarrow B_1$ be
H\"older continuous of order $\be$.  Then for any $a\le s<t\le b$, we have 
\begin{eqnarray}
&&|f(x_2(t))-f(x_1(t))-f(x_2(s)) +f(x_1(s))|_2 
\le \|f' \|_\infty \|  x_2-  x_1\|_{s, r, \be } (r-s)^\be \nonumber\\
&&\qquad \qquad +
\|f''\|_\infty \left[ \|x_1\|_{s,r, \be}+\|x_2\|_{s,r, \be}\right]\|
x_2-x_1\|_{s,r} (r-s)^\be  \,. \nonumber\\
\label{e.5.one-comp-holder}
\end{eqnarray}
\end{enumerate}
\end{lemma}
\begin{proof}  This inequality may be  well-known. 
We include  a short proof for the completeness. 
Using  the mean value theorem we have 
\begin{eqnarray*}
&& f(y_2)-f(y_1)-f(x_2)+f(x_1) \\
&&
=\int_0^1 f'( (1-\theta) y_1+\theta  y_2 )d\th  (y_2-y_1 )  \\
&&
-\int_0^1 f'((1-\theta) x_1+\theta x_2 )d\th  ( 
x_2-x_1)  \\
&&
=\int_0^1 f'((1-\theta) y_1+\theta  y_2 )d\th  (y_2-y_1-
x_2+x_1)  \\
&&  +\int_0^1\left[
f'((1-\theta) y_1+\theta  y_2 )- f'((1-\theta) x_1+\theta  x_2 ) \right] d\th (x_2-x_1)\\
&&=\int_0^1 f'(y_1+\theta (y_2-y_1))d\th  (y_2-y_1-
x_2+x_1)  \\
&&  +\int_0^1\int_0^1  
f''\left(\upsilon  \left[ y_1+
  \theta (y_2-y_1) \right] +(1-\upsilon)   
  \left[ (1-\theta) x_1+\theta  x_2 \right] \right)   d\th d\upsilon
  \\
  &&\qquad (x_2-x_1) \otimes  \left[ (1-\theta) (y_1-x_1 ) +\theta ( y_2   - x_2 ) \right]\,.   
\end{eqnarray*}
This proves \eqref{e.four-point}  easily.  The inequality 
\eqref{e.5.one-comp-holder} is  straightforward  consequences of 
\eqref{e.four-point}. 
\end{proof} 

\begin{lemma}\label{l.F-diff-holder-norm-bound} Denote 
\[
F^{(i)}_j( t)=F_j(t, \Ga_i, z_i)\,,\quad i, j=1,2\,. 
\]
Then, we have 
\begin{eqnarray}
&&\left\|\frac{d}{dt} F^{(2)}_1( t)-\frac{d}{dt}F^{(1)}_1(t)\right\|_\HH
  \le \kappa   \left[\|z_2 -z_1 \|_{0, \tau} +\|\Ga_2
 -\Ga_1 \|_{0, \tau}\right]    \nonumber\\  
\label{e.holder-f1}
\end{eqnarray}
and 
\begin{eqnarray}
&&\left\| F_{2}^{(2)} -F_{2}^{(1)}
\right\|_{a, b, \be}   
\le    
\kappa  (1+\|B\|_{a,b, \be}) \big(1+\|z_1\|_{0, a}+\|z_2\|_{0,a}+\|z_1\|_{a, b, \be}\nonumber\\
&&\qquad +
\|z_2\|_{a, b, \be} + \|\Ga_1\|_{a,b, \be}+\|\Ga_2\|_{a,b, \be}\big) 
\bigg[ |z_2(a)-z_1(a)|+\|\Ga_2(a)-\Ga_1(a)\|  \nonumber\\
&&\qquad +\left[ \|z_2-z_1\|_{a, b, \be} +\|  \Ga_2-  \Ga_1\|_{a, b, \be}\right]    
(b-a)^{\be} \bigg] \,. 
 \label{e.holder-f2}
\end{eqnarray}
\end{lemma}
\begin{proof}   Let $(\Ga_1, z_1)$ and $(\Ga_2, z_2)$  be two elements in
$\BB[0, \tau]$.    We recall 
\[
\begin{cases}
 F_1^{(i)}(t)= \om+\int_0^t  \tilde \si (s, z_i(s), \Gamma_i(s)+\om )\int_0^\cdot \phi(s,
u)du ds\,; \nonumber\\
 F_2^{(i)}(t)= \eta(\om ) +\int_0^t \tilde b(s, z_i(s), \Gamma_i(s)+\om )ds
+\int_0^t \si(s, z_i(s),
\Gamma_i(s)+\om )\de B(s)\nonumber\\ 
 \qquad\qquad    +\int_0^t\int_0^s  \si(s, z_i(s), \Ga_i(s)+\om ) \tilde \si(u, z_i(u),
\Ga_i(u)+\om ) \phi(s, u) duds\,.\nonumber
\end{cases} 
\]
To simplify notation we also denote
\begin{eqnarray*}
 b^{(i)}(s)&=&   b(s, z_i(s), \Gamma_i(s)+\om )\,, \quad 
\tilde b^{(i)}(s)= \tilde b(s, z_i(s), \Gamma_i(s)+\om )\,,\\
 \si ^{(i)}(s)&=& \si  (s, z_i(s), \Gamma_i(s)+\om )\,, \quad
\tilde \si ^{(i)}(s)=\tilde \si (s, z_i(s), \Gamma_i(s)+\om )\,.  
\end{eqnarray*} 
We have  for any $t\in [0, \tau]$, 
\begin{eqnarray*}
&&\left\|\frac{d}{dt} F^{(2)}_1( t)-\frac{d}{dt}F^{(1)}_1(t)\right\|_\HH
=\left\|\int_0^\cdot \left[ \tilde \si^{(2)}(t)-\tilde \si^{(1)}(t)\right] \phi(t, u) du\right\|_\HH \\
&&\qquad \le\kappa   \left\|\int_0^\cdot\left[ |z_2(t)-z_1(t)|  +\|\Ga_2
(t)-\Ga_1(t)\|_\HH \right]   \phi(t, u) du \right\|_\HH \\
&&\qquad \le\kappa   \left[\|z_2 -z_1 \|_{0, \tau} +\|\Ga_2
(t)-\Ga_1(t)\|_{0, \tau}\right]    \left[ \int_0^T \phi(t, u)^p du \right]^{1/p}   \\ 
&&\qquad \le \kappa   \left[\|z_2 -z_1 \|_{0, \tau} +\|\Ga_2
(t)-\Ga_1(t)\|_{0, \tau}\right]   \,. 
\end{eqnarray*}
This is \eqref{e.holder-f1}. 

As in the proof of Lemma \ref{l.uniform-bound-f1}   we   denote 
\[
\begin{cases}
F_{21}^{(i)}(t)=  \int_0^t \tilde b(s, z_i(s), \Gamma_i(s)+\om )ds\nonumber\\
F_{22}^{(i)}(t)= \int_0^t \si(s, z_i(s),
\Gamma_i(s)+\om )\de B(s)\nonumber\\  
F_{23}^{(i)}(t)=\int_0^t\int_0^s  \si(s, z_i(s), \Ga_i(s)+\om ) \tilde \si(u, z_i(u),
\Ga_i(u)+\om ) \phi(s, u) duds\,.\nonumber
\end{cases} 
\]
For any $a, b\in [0, \tau]$,  we have
\begin{eqnarray*}
&&|F_{21}^{(2)}(b)-F_{21}^{(1)}(b)-F_{21}^{(2)}(a)+F_{21}^{(1)}(a)|\nonumber\\
&&\quad =\int_a^b |\tilde b_2(r)-\tilde b_1(r)|dr\nonumber\\
&&\quad \le\kappa  \int_a^b\left[  |z_2(r)-z_1(r)  |+\|\Ga_2(r)-\Ga_1(r)\|_\HH\right] dr\nonumber\\
&&\quad \le \kappa  \left[  |z_2 -z_1   |_{a,b}+\|\Ga_2 -\Ga_1 \|_{a,b}\right] (b-a)\,. 
\end{eqnarray*}
This yields 
\begin{eqnarray}
\|F_{21}^{(2)} -F_{21}^{(1)}\|_{a, b, \be} 
&\le&  \kappa   \left[  |z_2 -z_1   |_{a,b}+\|\Ga_2 -\Ga_1 \|_{a,b}\right] (b-a)^{1-\be}
\nonumber\\
&\le&  \kappa   \left[ |z_2(a)-z_1(a)|+\|\Ga_2(a)-\Ga_1(a)\|\right](b-a)^{1-\be}\nonumber\\
&& +\kappa   \left[
 |z_2 -z_1   |_{a,b, \be }+\|\Ga_2 -\Ga_1 \|_{a,b, \be }\right] (b-a) \nonumber\\
&\le&  \kappa   \left[ |z_2(a)-z_1(a)|+\|\Ga_2(a)-\Ga_1(a)\|\right] \nonumber\\
&& +\kappa  \left[
 |z_2 -z_1   |_{a,b, \be }+\|\Ga_2 -\Ga_1 \|_{a,b, \be }\right] (b-a) ^\be 
 \,. 
\label{e.holder-f21}
\end{eqnarray}
 Now we find the bounds for  $F_{22}^{(i)}$.  We have 
\begin{eqnarray}
&&|F_{22}^{(2)}(b)-F_{22}^{(1)}(b)-F_{22}^{(2)}(a)+F_{22}^{(1)}(a)| \nonumber\\
&&\quad = \left|\int_a^b \left( \si^{(2)}(r) -\si^{(1)}(r) \right) \de B(r)\right|\nonumber\\
&&\quad =  \left|\int_a^b D_{b-}^{1-\al} B_{t-} (r)
D_{a+}^\al \left( \si^{(2)}(r) -\si^{(1)}(r) \right)   dr\right|\nonumber\\
&&\quad =  \frac{1}{\Ga(1-\al)}  \Bigg|\int_a^b  D_{b-}^{1-\al} B_{t-}(r)  
\Bigg( \frac{\si^{(2)}(r) -\si^{(1)}(r)}{(r-a)^\al} \nonumber\\
&&\qquad\quad 
+\al  \int_a^r \frac{\si^{(2)}(r) -\si^{(1)}(r)-\si^{(2)}(\rho) +\si^{(1)}(\rho)}{
(r-\rho)^{\al+1}} d\rho \Bigg)   dr\Bigg|\nonumber\\
&&\qquad 
\le \kappa \|B\|_{a, b, \be}\left[I_1+I_2\right]\,,\label{e.proof-bound-by-i1-i2}
\end{eqnarray}
where
\begin{eqnarray}
I_1
&=& \int_a^b   (b-r)^{\al+\be-1} 
  \frac{\left|\si^{(2)}(r) -\si^{(1)}(r)\right|}{(r-a)^\al}  dr\,; \label{e.def-i1} \\
  I_2
  &=& \int_a^b  (b-r)^{\al+\be-1}  \int_a^r \frac{
  \left|\si^{(2)}(r) -\si^{(1)}(r)-\si^{(2)}(\rho) +\si^{(1)}(\rho)\right| }{
(r-\rho)^{\al+1}} d\rho  dr\,. \nonumber\\ 
\label{e.def-i2} 
\end{eqnarray}
It is easy to see that
\begin{equation}
I_1\le \kappa  \left[ \|z_2-z_1\|_{a, b}+\|\Ga_2-\Ga_1\|_{a, b}\right] (b-a)^\be\,.
\label{e.diff-f2-i1}
\end{equation}
To bound $I_2$,  we need the following identity.
\begin{eqnarray}
\si^{(2)}(r)-\si^{(1)}(r)-\si^{(2)}(\rho)+\si^{(1)}(\rho)
=J_1+J_2+J_3+J_4+J_5\,,
\end{eqnarray}
where
\[
\begin{cases}
 J_1
= \si(r, z_2(r), \Ga_2(r)+\om )-\si(r, z_1(r), \Ga_2(r)+\om)\nonumber\\
 \qquad  \qquad -\si(r, z_2(\rho), \Ga_2(r)+\om)+\si(r, z_1(\rho), \Ga_2(r)+\om)\,; \nonumber\\
 J_2
= \si(r, z_1(r), \Ga_2(r)+\om)-\si(r, z_1(r), \Ga_1(r)+\om)\nonumber\\
\qquad  \qquad -\si(r, z_1(r), \Ga_2(\rho)+\om)+\si(r, z_1(r), \Ga_1(\rho)+\om)\,;\nonumber\\
 J_3
= \si(r, z_1(r), \Ga_2(\rho  )+\om)-\si(r, z_1(r), \Ga_1(\rho )+\om)\nonumber\\
 \qquad \qquad -\si(r, z_1(\rho), \Ga_2(\rho)+\om)+\si(r, z_1(\rho), \Ga_1(\rho)+\om)\,;\nonumber\\
 J_4
= \si(r, z_1(\rho), \Ga_2(\rho  )+\om)-\si(r, z_1(\rho), \Ga_2(r )+\om)\nonumber\\
\qquad  \qquad -\si(r, z_2(\rho), \Ga_2(\rho)+\om)+\si(r, z_2(\rho), \Ga_2(r)+\om)\,;\nonumber\\
 J_5
= \si(r, z_2(\rho), \Ga_2(\rho  )+\om)-\si(\rho, z_2(\rho), \Ga_2(\rho  )+\om)\nonumber\\
\qquad  \qquad  + \si(\rho, z_1(\rho), \Ga_1(\rho  )+\om)-\si(r, z_1(\rho), \Ga_1(\rho  )+\om)\,.  \nonumber 
\end{cases}
\]
From Lemma \ref{l.one-comp-holder}, we see that 
\begin{equation}
|J_1|\le \kappa \left[ \|z_2-z_1\|_{ \rho, r,  \be}  +\left(\|z_1\|_{\rho,r, \be}+
\|z_2\|_{\rho,r, \be}\right)\|z_2-z_1\|_{\rho, r}\right]  (r-\rho)^\be
\label{e.proof-j1}
\end{equation}
and 
\begin{equation}
\|J_2\|\le \kappa  \left[ \|  \Ga _2-  \Ga_1\|_{ \rho, r, \be }  +\left(\| \Ga_1\|_{\rho,r, \be}+
\|  \Ga_2\|_{\rho,r, \be }\right)\|\Ga_2-\Ga_1\|_{\rho, r, \be }\right]  (r-\rho)^\be \,. 
\label{e.proof-j2}
\end{equation}
Use the mean value theorem to obtain
\begin{eqnarray*}
J_3&=&\int_0^1\int_0^1 \DD \tilde \si( (1-\upsilon)z_1(\rho)+\upsilon 
 z_1(r), (1-\th)\Ga_1(\rho)+\th  \Ga_2(\rho)) d\th d\upsilon\nonumber  \\
&&\qquad  (z_1(r)-z_1(\rho)) (\Ga_2(\rho)-\Ga_1(\rho))\,. 
\end{eqnarray*}
This shows 
\begin{equation}
|J_3|\le  \kappa \|z_1\|_{\rho, r, \be}\|  \Ga_2-  \Ga_1\|_{\rho, r}  (r-\rho)^\be   \,. 
\label{e.proof-j3}
\end{equation}
In a similar way we can obtain
\begin{equation}
|J_4|\le  \kappa \|\dot \Ga_2\|_{ \rho, r }\|  z_2-  z_1\|_{\rho, r}  (r-\rho)    \,. 
\label{e.proof-j4}
\end{equation}
From the assumption on $\si$  it is to verify that  
\begin{equation}
|J_5|\le  \kappa   (r-\rho)  \left[ \|z_2-z_1\|_{ \rho, r}+\|\Ga_2-\Ga_1\|_{\rho, r}\right]   \,. 
\label{e.proof-j5}
\end{equation}
Combining \eqref{e.proof-j1}-\eqref{e.proof-j5}, we have
\begin{eqnarray}
&& \left|\si^{(2)}(r)-\si^{(1)}(r)-\si^{(2)}(\rho)+\si^{(1)}(\rho)\right|
\nonumber\\
&&\qquad \le \kappa 
  \bigg[\|z_2-z_1\|_{a, b, \be} +\|  \Ga_2-  \Ga_1\|_{a, b,\be }\nonumber\\
&&\qquad +\left(1+\|z_1\|_{a, b, \be}+
\|z_2\|_{a, b, \be}+\| \Ga_1\|_{a,b, \be }+\| \Ga_2\|_{a,b, \be }\right)      \|z_2-z_1\|_{a,b}\nonumber\\ 
&&\qquad +\left(1+\|z_1\|_{a, b, \be}+ \|z_2\|_{a, b, \be}+
\|  \Ga_1\|_{a, b, \be  }\right. \nonumber\\
&&\qquad \left. +\| \Ga_2\|_{a,b, \be }\right)      \|\Ga_2-\Ga_1\|_{a,b}\bigg]
(r-\rho)^\be \,.  
\end{eqnarray}
Substituting the above inequality into \eqref{e.def-i2}  yields 
\begin{eqnarray}
&&I_2 \le \kappa 
  \bigg[\|z_2-z_1\|_{a, b, \be} +\|  \Ga_2-  \Ga_1\|_{a, b, \be}\nonumber\\
&&\qquad +\left(1+\|z_1\|_{a, b, \be}+
\|z_2\|_{a, b, \be}+
\| \Ga_1\|_{a, b, \be } +\|  \Ga_2\|_{a,b, \be }\right)      \|z_2-z_1\|_{a,b}\nonumber\\ 
&&\qquad +\left(1+\|z_1\|_{a, b, \be}+\|z_2\|_{a, b, \be}+
\| \Ga_1\|_{a, b, \be }+\|  \Ga_2\|_{a,b, \be}\right)      \|\Ga_2-\Ga_1\|_{a,b}\bigg]
(b-a)^{2\be}  \,. \nonumber \\
\end{eqnarray}
Substituting the bounds for $I_1$ and $I_2$ into \eqref{e.proof-bound-by-i1-i2}
we have 
\begin{eqnarray}
&&\|F_{22}^{(2)}-F_{22}^{(1)}\|_{a,b,\be} 
\le    \kappa 
 \|B\|_{a,b, \be} \bigg[\|z_2-z_1\|_{a, b, \be} +\|  \Ga_2-  \Ga_1\|_{a, b, \be}\nonumber\\
&&\qquad +\left(1+\|z_1\|_{a, b, \be}+
\|z_2\|_{a, b, \be}+\|\Ga_1\|_{a,b, \be} +\|\Ga_2\|_{a,b, \be}\right)      \|z_2-z_1\|_{a,b}\nonumber\\ 
&&\qquad +\left(1+\|z_1\|_{a, b, \be}+\|z_2\|_{a, b, \be}+
\| \Ga_1\|_{a, b,\be }+\|\Ga_2\|_{a,b, \be}\right)      \|\Ga_2-\Ga_1\|_{a,b}\bigg]
(b-a)^{\be} \nonumber \\ 
&&\le    \kappa 
  \|B\|_{a,b, \be} \left(1+\|z_1\|_{a, b, \be}+
\|z_2\|_{a, b, \be}+ \|\Ga_1\|_{a,b, \be}+\|\Ga_2\|_{a,b, \be}\right) 
\nonumber\\
&&\qquad \bigg[ |z_2(a)-z_1(a)|+\|\Ga_2(a)-\Ga_1(a)\|  \nonumber\\
&&\qquad +\left[ \|z_2-z_1\|_{a, b, \be} +\|  \Ga_2-  \Ga_1\|_{a, b, \be}\right]    
(b-a)^{\be} \bigg]   \,.  
\label{e.holder-f22}
\end{eqnarray}
Finally,  we turn to bound $\|F_{23}^{2}-F_{23}^{1}\|_{a, b, \be} $.  
We have 
\begin{eqnarray*}
&&\left| F_{23}^{2}(b)-F_{23}^{1}(b)-F_{23}^{2}(a)+F_{23}^{1}(a)\right|\\
&&\quad \le \int_a^b \int_0^s \left|\si(s, z_2(s), \Ga_2(s)) -
\si(s, z_1(s), \Ga_1(s))  \right| |\tilde \si(u, z_2(u), \Ga_2(u))|duds\\
&&\qquad+\int_a^b \int_0^s \left|\tilde \si(u, z_2(u), \Ga_2(u)) -
\tilde \si(u, z_1(u), \Ga_1(u))  \right| |\si (s, z_1(s), \Ga_1(s))|duds\\
&&\quad \le \kappa (b-a) 
\left\{ \|z_2-z_1\|_{a, b} +\|\Ga_2-\Ga_1\|_{a, b} 
 +\|z_1\|_{a, b}\left[\|z_2-z_1\|_{0, b}+\|
 \Ga_2-\Ga_1\|_{0, b}\right]\right\}
 \,. 
\end{eqnarray*}
This means 
\begin{eqnarray}
&&\left\| F_{23}^{2} -F_{23}^{1}
\right\|_{a, b, \be}   \le    \kappa     (b-a) ^{1-\be} 
\big\{ \|z_2-z_1\|_{a, b} +\|\Ga_2-\Ga_1\|_{a, b} \nonumber \\
&&\qquad 
 +\|z_1\|_{a, b}\left[\|z_2-z_1\|_{0, b}+\|
 \Ga_2-\Ga_1\|_{0, b}\right]\big\}\nonumber\\
 &&\qquad \le    \kappa
 (1 +\|z_1\|_{a, b}) \left[\|z_2-z_1\|_{0, b}+\|
 \Ga_2-\Ga_1\|_{0, b}\right] (b-a) ^{1-\be} \nonumber\\
 &&\qquad \le   \kappa  (b-a) ^{1-\be} 
 (1 +\|z_1\|_{a, b}) \bigg[|z_2 -z_1 |_{0,a}+ \|
 \Ga_2 -\Ga_1 \| _{0, a} \nonumber\\
 &&\qquad +\left[
 \|z_2-z_1\|_{a, b, \be }+\|
 \Ga_2-\Ga_1\|_{a, b, \be }\right](b-a)^\be \bigg]\, . 
 \label{e.holder-f23}
\end{eqnarray} 
Combining  \eqref{e.holder-f21},  \eqref{e.holder-f22}, and 
\eqref{e.holder-f23},   we   prove \eqref{e.holder-f2}. 
\end{proof}

Now we are ready to prove  one of  our main theorems of this section.
\begin{theorem}\label{t.5.7}  Let $T\in (0, \infty)$ be any given  number. 
Assume the  hypothesis \ref{h.b-sigma}. 
Then,   the  equation \eqref{e.5.6} has a unique solution.  Moreover,  there is a $\tau_0>0$ such that 
the solution satisfies 
\begin{eqnarray}
\sup_{0\le t\le T}|z(t)|
&\le&  c_2 \exp\left\{ c_1 \|B\|_{0, T, \be}^{1/\be}\ 
\right\} \label{e.final-uniform-bound-solution} \\
\sup_{0\le a< b\le T, b-a\le \tau_0}|z |_{a, b, \be} 
&\le & c_2 \exp\left\{ c_1 \|B\|_{0, T, \be}^{1/\be} 
\right\}\,\label{e.final-holder-bound-solution} 
\end{eqnarray}
for some constants $c_1$ and $c_2$ dependent only on 
$\be,     p, T$ and the bound $\cL$  for the coefficients $b$ and $\si$.  
\end{theorem}
\begin{proof} Let $\BB=\HH\oplus \RR$. Then $F=(F_1,   F_2)$ defined by 
\eqref{e.def-F} is a mapping from (some domain of) $\BB[0, T]$ to $\BB[0, T]$.  
The inequality \eqref{e.5.17} implies
that  there is a constant $c$ depending only on $\be,     p, T$ and the bound $\cL$  for the coefficients $b$ and $\si$ such that 
\[
\|F_1\|_{a, b, \be}\le c
\]
for any $a, b\in [0, T]$.  This together with \eqref{e.sup-f2}  implies that 
$F=(F_1, F_2)$ satisfies the condition (i)  of Theorem \ref{t.fix-point-theorem}
with $\kappa$ there being replaced by $c   \left(1+ \|B\|_{a, b, \be}\right)$. 
Lemma \ref{l.F-diff-holder-norm-bound} implies that 
$F=(F_1, F_2)$ satisfies the condition (ii) of Theorem 
\ref{t.fix-point-theorem}
with  the function $\bar h$ being  given by 
\begin{eqnarray*}
\bar h&=&\bar h((\Ga_1,z_1), (\Ga_2,z_2))
=
\cL (1+\|B\|_{a,b, \be}) \big(1+\|z_1\|_{0, a}+\|z_2\|_{0,a}\\
&&\qquad +\|z_1\|_{a, b, \be}  +
\|z_2\|_{a, b, \be} + \|\Ga_1\|_{a,b, \be}+\|\Ga_2\|_{a,b, \be}\big) \,.
\end{eqnarray*}
Thus,   we can apply Theorem \ref{t.fix-point-theorem} to prove that  there is a $x
\in \BB[0, T]$ satisfies the equation  $x(t)=F(t, x )$.  This means  that  $x$ satisfies the equation
\eqref{e.to-solve}, hence it satisfies \eqref{e.5.6}.    The bounds 
\eqref{e.final-uniform-bound-solution}  and 
\eqref{e.final-holder-bound-solution}  are   immediate consequence of 
\eqref{e.bound-holder-x-theorem}. 
\end{proof}

\subsection{Solution to the original equation}
We need the following lemmas in the proof of the existence and uniqueness theorem
for equation \eqref{e.5.6}.  
\begin{lemma}\label{l.diff-of-La}  Let  $\Ga :\cT\times \Om\rightarrow \Om$ be continuously 
differentiable in $t$ and $\HH$-differentiable in $\om$.  If $\Ga(t):\Om\rightarrow \Om$ has an inverse
$\La(t)$ and if $\La(t)$ is differentiable in $t$ in the Hilbert space
$\HH$,   then 
\begin{equation}
\frac{\partial \La}{\partial t}(t, \om)
=-(\DD \La)(t, \om)\frac{\partial
\Ga}{\partial t}(t, \La(t, \om))\,.
\end{equation}
\end{lemma}

\begin{proof} Since $\La(t)$ is the inverse of $\Ga(t)$ we have
\[
\Ga(t, \La(t, \om))=\om\,, \quad  \forall   \ \om \in \Om\,.
\]
Differentiating  both sides with respect to $\om $, we have
\[
(\DD  \Ga)(t, \La(t, \om))(\DD  \La)(t, \om)=I\,,
\]
where $I$ is an identity operator from $\HH$ to $\HH$.
Therefore, we obtain
\begin{equation}
\left[(\DD  \Ga)(t, \La(t, \om))\right]^{-1}=(\DD  \La)(t, \om)\,.
\label{e.proof-nabla-G}
\end{equation}
On the other hand, differentiating $\Ga(t, \La(t, \om))=\om$ with respect to $t$, we
have
\[
\frac{\partial \Ga}{\partial t}(t, \om)
+(\DD  \Ga)(t, \La(t)(\om))\frac{\partial \La}{\partial t}(t, \om)=0\,.
\]
Thus
\begin{equation}
\frac{\partial \La}{\partial t}(t, \om)
=-\left[(\DD  \Ga)(t, \La(t, \om))\right]^{-1}\frac{\partial
\Ga}{\partial t}(t, \La(t, \om))\,.
\label{e.proof-diff-La} 
\end{equation}
 
Combining   \eref{e.proof-nabla-G}  and \eref{e.proof-diff-La}    we have
\begin{equation}
\frac{\partial \La}{\partial t}(t, \om)
=-(\DD  \La)(t, \om)\frac{\partial
\Ga}{\partial t}(t, \La(t, \om))\,.
\end{equation}
This proves the lemma. 
\end{proof}

\begin{lemma}\label{l.reduced-to-original}
Let $\tau\in (0, T]$ be a positive number.  
Assume that   $\Ga(t):\Om\rightarrow \Om$ defined by \eqref{e.5.6} has an inverse
  $\La(t)$ for all $t\in [0, \tau]$  and assume that $\La(t)$ is differentiable in $t\in [0, \tau]$ in the Hilbert space
$\HH$.  Let $z$ be defined  by  \eqref{e.5.6}.    Then $x(t)=z(t, \La(t)), t\in [0, \tau]$
satisfies the equation \eqref{e.5.5}.  
\end{lemma} 
\begin{proof}
If $\Ga(t)$ defined by \eqref{e.5.6} has inverse $\La(t)$,
then 
\begin{eqnarray*}
\om&=&
\Gamma(t, \La(t))=\La(t)+\int_0^t  \tilde \si (s, z(s),
 \Gamma(s))\int_0^\cdot \phi(s,
u)du ds\Big|_{\om=\La(t)}\,. 
\end{eqnarray*}
 Or
\begin{equation} 
\begin{cases}
 \La(t) =\om+\int_0^\cdot  h(t,u,\om) du\qquad \quad \hbox{with} \\
 h(t,u, \om)
=- \int_0^t  \tilde \si (s, z(s),
 \Gamma(s))  \phi(s,
u)  ds\Big|_{\om=\La(t)}\,.  
\end{cases}
\label{e.h-for-La}
\end{equation}
On the other hand, from \eqref{e.5.6}  we   have
\begin{eqnarray}
\frac{d}{dt}\Ga(t,\om)\big|_{\om=\La(t)}
&=& \tilde \si(t, z(t), \Ga(t))\int_0^\cdot \phi(t, u) du\big|_{\om=\La(t)}\nonumber\\
&=& \tilde \si(t, x(t,\om), \om)\int_0^\cdot \phi(t, u) du \,. \label{e.diff-Ga-at-La} 
\end{eqnarray}

We apply the It\^o     formula \eqref{e.ito-formula} to
$z(t, \La(t))$ with
\begin{eqnarray*}
f_0&=& \tilde b(s, z(s), \Gamma(s)) 
 +   \si(s, z(s), \Ga(s)) \int_0^s \tilde \si(u, z(u),
\Ga(u)) \phi(s, u) du \\
f_1&=&   \si(s, z(s),
\Gamma(s))\ 
\end{eqnarray*}
and with $h$ defined by \eqref{e.h-for-La}. We shall use $\si(s, x(s))$ to 
denote $\si(s, x(s, \om), \om)$ etc. 
Noticing 
$z(s)\big|_{\om=\La(s)}=x(s)$, we have
\begin{eqnarray}
x(t)
&=& z(t, \La(t))\nonumber\\
&=& \eta(\om)
 +   \int_0^t  \left(\int_0^s \tilde \si(u, z(u),
\Ga(u)) \phi(s, u) du \right) \Big|_{\om=\La(s)} \si(s, x(s) )  ds\nonumber\\
&&\quad +\int_0^t \tilde b(s, x(s) ) ds + \int_0^t \si (s, x(s))\de B(s) \nonumber\\
&&\quad +\int_0^t  \DD   z(s, 
\La(s))\frac{d}{ds}\La(s) ds +\int_0^t \si(s, x(s)) h(s,s,\om) ds\nonumber\\
&=&\eta(\om)+I_1+I_2+I_3+I_4+I_5 \,. \label{e.i1-4-for-x}
\end{eqnarray}
From \eqref{e.h-for-La} and since $\phi(s,u)=\phi(u,s)$, we see 
that 
\begin{equation}
I_1+I_5=0\,.\label{e.i1-i4}
\end{equation}  
By Lemma \ref{l.diff-of-La} and then by \eqref{e.diff-Ga-at-La},  we have
\begin{eqnarray*}
I_4
&=&\int_0^t (\DD  z)(s, \La(s))\frac{\partial }{\partial s} \La(s) ds\\
&=& -\int_0^t (\DD  z)(s, \La(s)) (\DD  \La)(s)\left(\frac{\partial }{
\partial s}\Ga\right)
(s, \La(s))ds\\ 
&=& -\int_0^t  \left[ \DD  x (s)\right] \tilde \si(s , x(s))   \int_0^\cdot \phi(s,u) du  ds \,.
 \end{eqnarray*}
This yields
\begin{equation}
I_4=-\int_0^t  \tilde \si(s , x(s)) \DD^\phi_s x(s) ds\,.\label{e.i3}  
\end{equation}
Substituting \eqref{e.i1-i4} and  \eqref{e.i3} into \eqref{e.i1-4-for-x} we have 
\begin{eqnarray}
x(t)
&=&\eta(\om)
 +     \int_0^t \tilde b(s, x(s) ) ds + \int_0^t \si(s, x(s))\de B(s)\nonumber\\
 &&\qquad 
 -\int_0^t  \tilde \si(s , x(s)) \DD^\phi_s x(s) ds\,. 
\end{eqnarray}
This is the lemma.  
\end{proof}
Now we show that there is a positive $\tau$ such that $\Ga(t)$ has inverse 
$\La(t)$.  Before we continue, we need the following simple inequality. 
\begin{lemma}Assume that $
B:[0, T]\rightarrow \RR$ is a H\"older continuous function of exponent 
$\be \in (0,1)$.    
Let $\BB_1 $ and $\BB_2$ be two  Banach spaces and let
$f:[0, T]\rightarrow \BB_1 $ and $g:[0, T]\rightarrow \BB_2$ be two 
H\"older continuous functions with exponent $\al  \in (  1-\be, 1)$. 
Then 
\begin{eqnarray}
&&\|\int_0^\cdot f(s)\otimes g(s) dB(s)\|_{a, b, \be}
 \le   \kappa \|B\|_{a, b, \be}   
\bigg\{ \|f\|_{a, b}\|g\|_{a,b}  \nonumber \\
&&\qquad\qquad\qquad    +
\left[\|f\|_{a, b}\|g\|_{a, b, \be}+\|g\|_{a, b}\|f\|_{a, b, \be}\right] (b-a)^{\be }
\bigg\} \,.
\label{e.product-holder-bound}
\end{eqnarray}
\end{lemma} 
\begin{proof} 
We refer to \cite{hugaussian},  and in particular, the references therein
for the tensor product.  Since $\al+\be>1$, we can choose a $\la$ such that
$\la<\al$ and $1-\la<\be$.  
For any $a, b\in [0, T]$,  we have
\begin{eqnarray*}
\left\|\int_a^b f(s)\otimes g(s) dB(s) \right\|
&=&  \left\| \int_a^b  
  D_{a+}^{\la} \left[f(s)\otimes g(s)\right] D_{b-}^{1-\la} B_{b-}(s) ds
\right\|    \\
&\le &\kappa \|B\|_{a, b, \be}  \int_a^b (b-s)^{\be+\la-1} 
 \left\| D_{a+}^{\la} \left[f(s)\otimes g(s)\right]
\right\|   ds\,. 
\end{eqnarray*}
From the definition of the Weyl derivative \eqref{e.weyl-der-left}, we see essily
\begin{eqnarray*}
&&\left\| D_{a+}^{\la} \left[f(s)\otimes g(s)\right]
\right\|\le \kappa \bigg\{ \|f\|_{a, b}\|g\|_{a,b} (s-a)^{-\la} \\
&&\qquad \qquad +
\left[\|f\|_{a, b}\|g\|_{a, b, \be}+\|g\|_{a, b}\|f\|_{a, b, \be}\right] (s-a)^{\be-\la}
\bigg\}\,. 
\end{eqnarray*}
Therefore, we have
\begin{eqnarray*}
&&\left\|\int_a^b f(s)\otimes g(s) dB(s) \right\| 
\le \kappa \|B\|_{a, b, \be}  \int_a^b (b-s)^{\be+\la-1} 
\bigg\{ \|f\|_{a, b}\|g\|_{a,b} (s-a)^{-\la} \\
&&\qquad   +
\left[\|f\|_{a, b}\|g\|_{a, b, \be}+\|g\|_{a, b}\|f\|_{a, b, \be}\right] (s-a)^{\be-\la}
\bigg\}    ds\\
&&\qquad \le \kappa \|B\|_{a, b, \be}   
\bigg\{ \|f\|_{a, b}\|g\|_{a,b}   \\
&&\qquad   +
\left[\|f\|_{a, b}\|g\|_{a, b, \be}+\|g\|_{a, b}\|f\|_{a, b, \be}\right] (b-a)^{\be }
\bigg\} (b-a)^\be 
\end{eqnarray*} 
which implies the lemma. 
\end{proof}

By a Picard iteration procedure and by the bounds that we are going to 
obtain, we can show that $\DD  \Ga(t )$ and $\DD  z(t )$ exist
under the hypothesis \ref{h.b-sigma}. 

From the equation \eqref{e.5.6},  we have
\begin{eqnarray*}
\DD  \Ga(t)
&=&I+\int_0^t \si_{xx}(s, z(s), \Ga(s)) \DD  z(s) \otimes \int_0^\cdot 
\phi(s, u) du ds\\
&&\qquad +\int_0^t \DD  \si_{x }(s, z(s), \Ga(s)) \DD  \Ga(s) \otimes \int_0^\cdot 
\phi(s, u) du ds 
\end{eqnarray*}
and 
\begin{eqnarray*}
\DD  z(t)
&=&\DD  \eta+\int_0^t \tilde b_{x }(s, z(s), \Ga(s)) \DD  z(s)   ds+
\int_0^t \DD   \tilde b (s, z(s), \Ga(s)) \DD  \Ga(s)   ds\\
&&  +\int_0^t \si_{x }(s, z(s), \Ga(s)) \DD  z(s)  \de B(s)\\
&&+
\int_0^t \DD   \si (s, z(s), \Ga(s)) \DD  \Ga(s)   \de B(s)    +\si(\cdot, z(\cdot), \Ga(\cdot)) I_{[0, t]}(\cdot)\nonumber\\
&&+
\int_0^t \si_{x }(s, z(s), \Ga(s)) \DD  z(s)  \int_0^s  \si_{x }(u, z(u), \Ga(u ))  
\phi(s,u) du ds   \\
&& +
\int_0^t \DD   \si (s, z(s), \Ga(s)) \DD  \Ga(s)  \int_0^s  \si_{x }(u, z(u), \Ga(u ))  
\phi(s,u) du ds  \\
&& +
\int_0^t   \si (s, z(s), \Ga(s))    \int_0^s  \si_{xx }(u, z(u), \Ga(u ))  \DD  z(u)
\phi(s,u) du ds\\
&&+
\int_0^t   \si (s, z(s), \Ga(s))    \int_0^s  \DD  \si_{x  }(u, z(u), \Ga(u ))  \DD  \Ga(u)
\phi(s,u) du ds\,. 
\end{eqnarray*}
Similar to the lemmas \ref{l.uniform-bound-f1} and \ref{l.uniform-bound-f2}, we can obtain 
\begin{lemma}\label{l.bounds-for-nablas} Under the hypothesis
\ref{h.b-sigma},   there is a $\kappa_B$ depending on  $\be, p, T$
and $\|B\|_{0, T, \be}$ such that 
\begin{eqnarray}
\|\frac{d}{dt} \DD  \Ga \|_{0, t}
&\le&  \kappa_B \left(\|\DD  z \|_{0, t} +\|\DD  \Ga \|_{0, t} \right)\,;
\label{e.5.58}\\
\|\DD  z\|_{a, b, \be}
&\le& \kappa_B \big(1+\|\DD  z\|_{0, a} +\|\DD  \Ga\|_{0, a}\nonumber\\
&&\quad +\|\DD  z\|_{a, b, \be} (b-a)^\be +\|\DD  \Ga\|_{a, b, \be} (b-a)^\be \big)\,.\label{e.5.59}
\end{eqnarray} 
\end{lemma}
\begin{proof}Let us denote the integral terms in the above expression
for $\DD  z(t)$ by $I_k$, $k=1, 2, \cdots, 8$.   Let us explain how to bound 
$I_3=\int_0^t \si_{x }(s, z(s), \Ga(s)) \DD  z(s)  \de B(s)$.  
Since $z(s)$, $\Ga(s)$ are H\"older continuous
(of exponent $\be$ with respect to $s$), $f(s):=\si_{x }(s, z(s), \Ga(s))$ is then also
H\"older continuous. Now the inequality \eqref{e.product-holder-bound} can be invoked to obtain
\begin{eqnarray}
\|I_3\|_{a, b, \be}
&\le& \kappa \|B\|_{a, b, \be}   
\bigg\{ \|f\|_{a, b}\|\DD  z\|_{a,b} \nonumber \\
&&\qquad  +
\left[\|f\|_{a, b}\| \DD  z\|_{a, b, \be}+\|\DD  z\|_{a, b}\|f\|_{a, b, \be}\right] (b-a)^{\be }
\bigg\} \nonumber\\
&\le& \kappa_B    
\bigg\{  \|\DD  z\|_{a,b} +
\left[  \|\DD  z\|_{a, b, \be}+ \|\DD  z\|_{a, b, \be}\right] (b-a)^{\be }
\bigg\}
\end{eqnarray}
since  both $\|f\|_{a, b}\|$ and $\|f\|_{a, b, \be}$ are bounded by 
$\kappa_B$.  The other terms can be treated in exactly the same way.  
\end{proof}  
\begin{lemma}\label{l.nabla-gamma-z-bounds} Under the hypothesis,    there is a $\kappa_B$ depending on  $\be, p, T$
and $B_{0, T, \be}$ such that  
\begin{equation}
\|\DD  \Ga \|_{0, T}+\|\DD  z\|_{0, T}\le \kappa_B\,. 
\label{e.uniform-bound-nablas}
\end{equation}
\end{lemma}
\begin{proof}  To show the existence of $\DD  \Ga(t)$ and $\DD  z(t)$ 
and to show the above bound 
\eqref{e.uniform-bound-nablas} we still use the idea in the proof of Theorem 
\ref{t.fix-point-theorem}.   First, we can show that the $x_n$ defined there 
are $\HH$-differentiable and   similar bounds holds as those in  Lemma 
\ref{l.bounds-for-nablas} hold recursively for all $x_n$.  This can be used
to obtain the uniform bounds for all $n$.   Since the proof is analogous
that of Theorem \ref{t.fix-point-theorem}, we shall not provide it here again. 
\end{proof}

\begin{lemma}\label{l.5.13}  Let $G:\Om\rightarrow \HH$ be 
continuously $\HH$-differentiable mapping such that 
\[
\|\DD  G\|_\HH \le c<1 \,. 
\]
Define $\Ga:\Om\rightarrow \Om$ by 
\[
\Ga(\om)=\om+G(\om)\,.
\]
Then $\Ga$ has a (unique) inverse $\La$ such that 
$\Ga(\La(\om))=\La(\Ga(\om))=\om$.  
\end{lemma}
\begin{proof}
We define 
\[
\La_0(\om)=\om\,,\quad \La_{n+1}(\om)=\om-G(\La_n(\om))\,,
\quad n=0, 1, 2, \cdots\,.
\]
Then 
\begin{eqnarray*}
\left\|\La_{n+1}(\om)-\La_n(\om)\right\|_\HH
&=&\left\|G(\La_n(\om))-G(\La_{n-1}(\om))\right\|_\HH\\
&\le& c \left\|\La_{n }(\om)-\La_{n-1}(\om)\right\|_\HH\le \cdots \\
&\le& c ^n \left\|\La_{1 }(\om)-\La_{0}(\om)\right\|_\HH\,.
\end{eqnarray*}
This means that 
\[
\La_n(\om)-\om=\sum_{k=1}^n \left(\La_k(\om)-\La_{k-1}(\om)\right) 
\]
is a Cauchy sequence in $\HH$.  Thus $\La_n(\om)$ converges
to an element $\La(\om) $   in $\Om$. 
From the construction of $\La_n$ we see that $\La$ satisfies
\[
\La(\om)=\om -G(\La(\om))\,.
\]
Thus, 
\begin{eqnarray*}
\Ga(\La(\om))
&=& \La(\om)+G(\La(\om))\\
&=& \om -G(\La(\om))-G(\La(\om))=\om\,.
\end{eqnarray*}
This   exactly  means that $\La$ is the inverse of $\Ga$.  
\end{proof}

Now we can state the main theorem of this paper.
\begin{theorem}\label{t.5.14}  Let $b, \si:[0, T]\times \RR\times \Om\rightarrow 
\RR$ satisfy the hypothesis \eqref{h.b-sigma}.  Then, the equation 
\eqref{e.5.1} has a unique solution $x(t)$ up to some positive random  time  $\tau>0$.  This means that there is a unique  positive  random  time $\tau>0$ such that 
\begin{eqnarray}
x (t\wedge\tau ) = \eta+\int_0^{t\wedge\tau } 
 b(s, x (s), \om)d  s+
 \int_0^{t\wedge\tau } 
\si (s, x (s), \om)d B(s)   \,,\
  \forall \ t\in [0, T]\,. \nonumber\\
 \label{e.5.88}
\end{eqnarray}
\end{theorem} 
\begin{proof} 
 According to Lemma \ref{l.reduced-to-original}  the remaining main task   is to show the existence 
of an  inverse $\La(t)$  of $\Ga(t)$.  Since the existence and uniqueness of the solution of the system \eqref{e.5.6} is known by Theorem
\ref{t.5.7},  we only need to  consider the first equation of the system \eqref{e.5.6}, which is 
\[
\Gamma(t)=\om+\int_0^t  \tilde \si (s, z(s), \Gamma(s))\int_0^\cdot \phi(s,
u)du ds\,.
\]
Due to the presence of
$z(t)$ in the above equation it is hard to prove that $\Ga(t)$ has 
an inverse for all $t\in [0, T]$.   In fact, since $z(t)=z(t, \om)$ has been proved to exist, we may write the above equation as
\[
\Gamma(t)=\om+\int_0^t \hat  \si (s, \om, \Gamma(s))  ds\,,
\]
where  $\hat \si(s,\om, \Ga)=   \tilde \si (s, z(s, \om), \Gamma )\int_0^\cdot \phi(s,
u)du ds$ is mapping from $[0, T]\times \Om\times \Om\rightarrow \HH$.  
[As earlier we may replace $\Ga(t)-\om$ by $\Ga(t)$ so that $\hat \si$ is a 
mapping from $[0, T]\times \Om\times \HH\rightarrow \HH$.] 
Or  we can write the following differential equation in $\HH$:  
\[
\dot \Gamma(t)=  \hat \si (t, \om, \Gamma(t))   \,.
\]
The dependence on $\om$ in the coefficient $\hat  \si$ may prevent the solution
$\Ga(t)$ to have inverse for all time $t$. To explain we give one example 
in one dimension ($\dim(\Om)=\dim(\HH)=1$).  We let $\hat \si(t, \om, \Ga)=
-\om -\Ga$.  Then the solution to the equation with initial condition
$\Ga(0)=\om$ is explicitly given by $\Ga(t,\om)=2\om e^{-t} -\om$.  But 
$\Ga(t, \om)=0$  when $t= \ln 2$.  So,  $\Ga(t, \om)$  is not invertible
when $t= \ln 2$.   Due to the above example, we are only seeking the inverse 
of $\Ga(t)$ when $t$ is sufficiently small (but strictly positive). 

We shall prove that $\Ga(t):\Om\rightarrow \Om$ has an inverse when $t$ is sufficiently small. Given an arbitrarily  fixed positive number $R$,  we define the following  random  time:  
\begin{equation}
\tau_R=\tau_R(\om)=\inf\left\{ t>0\,, \quad |B(t, \om)|>R\,, 
\|B\|_{0,t, \be }>R\right\}\,.
\end{equation} 
Since  we have chosen a version of the Brownian motion $B$ such that
$B(t)$ is H\"older continuous, we see that  $0<\tau_R<\infty$ for all $\om\in \Om$.  

Now we define  
\begin{equation}
B_R(t,\om)=\begin{cases}
B(t,\om)&\qquad \quad \hbox{when \ $0\le t\le \tau_R$}\\
B(\tau_R, \om) &\qquad \quad \hbox{when \ $  t\ge \tau_R$}\,.
\end{cases}
\end{equation}
 
It is clear that
\[
\sup_{0\le t\le T} |B_R(t)|\le R
\quad{\rm  and}\quad \|B_R\|_{a, b,\be}\le 2R
\quad \hbox{ for any $0\le a<b\le T$}\,.
\]
Now we consider the equation \eqref{e.5.6} with $B$ replaced by
$B_R$ and the corresponding solutions are  replaced by $\Ga_R$ and $z_R$: 
\begin{equation}
\begin{cases}
\Gamma_R(t)=\om+\int_0^t  \tilde \si (s, z_R(s), \Gamma_R(s))\int_0^\cdot \phi(s,
u)du ds\,;\\ \\
z_R(t)= \eta(\om ) +\int_0^t \tilde b(s, z_R(s), \Gamma_R(s))ds
+\int_0^t \si(s, z_R(s),
\Gamma_R(s))\de B_R(s)\\ 
\qquad\qquad  +\int_0^t\int_0^s  \si(s, z_R(s), \Ga_R(s)) \tilde \si(u, z_R(u),
\Ga_R(u)) \phi(s, u) duds\,. 
\end{cases} \label{e.5.64}
\end{equation}
By the inequality \eqref{e.uniform-bound-nablas} we see that
\begin{equation*}
\|\DD  \Ga_R \|_{0, T}+\|\DD  z_R\|_{0, T}\le c_R \,, 
\end{equation*}
where $ c_R$  is a constant independent of $B$ (then independent of $\om$).   This combined with Lemma \ref{l.bounds-for-nablas} yields
\begin{equation}
\|\frac{d}{ds}\DD  \Ga(s)\|_{0, T}\le c_R\,.
\label{e.5.67}
\end{equation}  
On the other hand, we  can write 
\[
\DD  \Ga(t)=I+G(t,\om)\,,\quad\hbox{where}\quad 
G(t,\om)=\int_0^t \frac{d}{ds}\DD  \Ga(s) ds\,. 
\]
The inequality \eqref{e.5.67} implies that 
\[
\|\DD G(t, \om)\|\le c_Rt\le 1/2\,,\quad \hbox{if}\ \ 
t\le 1/(2c_R)\,.
\] 
Thus, from Lemma \ref{l.5.13} it follows that 
  $\Ga(t):\Om\rightarrow \Om$ has 
a (unique)  inverse $\La(t)$ when $t\le t_0:=1/(2c_R)$. 
Thus,  the system of equation \eqref{e.5.64}
 has a unique solution such that $\Ga_R(t)$ has a
 (unique) inverse $\La_R(t)$.  By Lemma 
 \ref{l.reduced-to-original} we see that
 $x_R(t)=z_R(t, \La_R(t))$ is a solution to
 \begin{eqnarray*}
x_R(t)&=&\eta+\int_0^t \tilde b(s, x_R(s), \om)ds+
\int_0^t \si(s, x_R(s), \om)\de B_R(s)\nonumber\\
&&\qquad 
 -\int_0^t \tilde \si (s, x(s), \om)\DD^\phi_s x_R(s) ds\,,
 \quad 0\le t\le t_0\,. 
  \end{eqnarray*}
But when $t\le  \tau_R$, $B_R(t)=B(t)$.  Then
when $t\le  t_0\wedge\tau_R$, we have
\begin{eqnarray*}
x_R(t)&=&\eta+\int_0^t \tilde b(s, x_R(s), \om)ds+
\int_0^t \si(s, x_R(s), \om)\de  B(s)\nonumber\\
&&\qquad 
 -\int_0^t \tilde \si (s, x(s), \om)\DD^\phi_s x_R(s) ds\,,
 \quad 0\le t\le t_0\wedge\tau_R\,. 
  \end{eqnarray*}
This can also be written as
\begin{eqnarray*}
x_R(t\wedge\tau_R)&=&\eta+\int_0^{t\wedge\tau_R} 
 b(s, x_R(s), \om)d  s+
 \int_0^{t\wedge\tau_R} 
\si (s, x_R(s), \om)d B(s)  \,. 
  \end{eqnarray*}
The theorem is then proved. 
\end{proof}

\section{Linear and quasilinear  cases}
\subsection{Quasilinear case} 
If the diffusion coefficient $\si$ satisfies 
\begin{equation}
\si(s, x, \om)=a_1(s, \om) x+a_0(s, \om) \,,
\label{e.si-quasilinear} 
\end{equation}
then we say equation \eref{e.ito-equation}  is {\it quasilinear stochastic
differential equation }driven by fractional Brownian motion.
Associated with this equation, 
the corresponding system    of equations  \eref{e.5.6} can be  written as
\begin{equation}
\begin{cases}
\Gamma(t)=\om+\int_0^t  a_1 (s, \Gamma(s))\int_0^\cdot \phi(s,
 u)du ds\\ \\
z(t)=\eta+\int_0^t \tilde b(s, z(s), \Gamma(s))ds +\int_0^t \si(s, z(s),
\Gamma(s))\de B(s)\\
\ \qquad \qquad +\int_0^t \si(s, z(s), \Ga(s)) \int_0^s a_1(u, \Ga(u)) \phi(s, u) du ds\,,
\end{cases}
\label{e.quasilinear.transformed} 
\end{equation}
where 
\begin{equation}
\tilde b(t,x,\om)=\tilde b(s, x, \om):=b(s, x, \om)- x \DD^\phi_s  \si_1(s,   \om) -\DD^\phi_s  \si_0 
(s,   \om)\,. 
\end{equation} 
This system is decoupled. We  can first  solve the above first equation.  
\begin{proposition}
Assume  that $a_1(t,\om)$ is uniformly Lipschitz in $\om$ with respect to  
$\HH$ norm. Namely,   there is a positive constant $\kappa$ such that 
\begin{equation}
|a_1(t, \om+h)-a_1(t, \om)|\le \kappa \|h\|_\HH\,, \qquad
\forall \ t\in [0, T]\,,\ \, \om\in \Om\,, \ h\in \HH\,. 
\end{equation}
 Then the first equation in \eref{e.quasilinear.transformed}   has a unique solution 
 $\Ga(t)$.  For all $t\in [0, T]$,  $\Ga(t):\Om\rightarrow \Om$ has an inverse
 $\La(t)$.  Moreover, 
the inverse $\La(t)$ is given by 
 $\La(t)=\La(t,t)$, where  $\left\{\La(\cdot, t)\right\}$ satisfies 
 \begin{equation}
  \La(s,t)=\om +\int_0^s  a_1 (t-v, \La(v, t))\int_0^\cdot \phi(v,
u)du dv\,, \quad 0\le s\le t\,. 
\label{e.inverse-gamma} 
 \end{equation}
\end{proposition}
\begin{proof} From the general dynamic system theory we see that
for any $t_0\in [0, T]$,  there is a unique solution $\Ga(t)=\Ga(t, t_0, \om)$
such that the first equation of \eqref{e.quasilinear.transformed} has a unique solution
for all $t\in [0, T]$ (even when $t<t_0$) such that $\Ga(t_0, t_0, \om)=\om$ and
$ \Ga(t, t_0, \om)$ satisfies the flow property:
\[
\Ga(t, s, \Ga(s, t_0, \om))=\Ga(t+s, t_0, \om)\,,  \
\forall \ t_0, s, t\in [0, T]\ \hbox{such that $s+t\in [0, T]$}\,. 
\]
This can be used to show the proposition easily. 
\end{proof} 
\def\bb{\mathfrak{b}}  
Once we obtain $\Gamma(t, \om)$ we can substitute it into the second
equation in \eref{e.quasilinear.transformed} to obtain the following equation
\begin{equation}
z(t)=\eta+\int_0^t \bb (s, z(s), \om)ds +\int_0^t \si(s, z(s),
\Gamma(s))\de B(s)\,, \label{e.quasilinear6.6}
\end{equation}
where
\begin{eqnarray*}
\bb (s, z , \om)
&=& \tilde b(s, z, \Ga(s))+\si(s, z , \Ga(s)) \int_0^s a_1(u, \Ga(u)) \phi(s, u) du \\
&=& b(s, x, \Ga(s))- x \DD^\phi_s  \si_1(s,   \Ga(s)) -\DD^\phi_s  \si_0 
(s,   \Ga(s))\\
&&\qquad 
+\si(s, z , \Ga(s)) \int_0^s a_1(u, \Ga(u)) \phi(s, u) du\,. 
\end{eqnarray*}
Now we can use Lemma \ref{l.reduced-to-original} to obtain the following theorem.
\begin{theorem} Let  the diffusion coefficient $\si$  be given by 
\eqref{e.si-quasilinear}.  Let $\Ga$ be the unique solution to the first 
equation of \eqref{e.quasilinear.transformed}.  Let $z$ be the unique solution
to the second equation of  \eqref{e.quasilinear.transformed}.
Then \eqref{e.ito-equation} has a unique solution $x(t)$ which is given by
\begin{equation}
x(t)=z(t, \La(t))\,. 
\end{equation} 
Moreover,    there are   positive constants $c_1$ and $c_2$, $\De\in [0, T]$, 
depending only on $p, \be, T$ such that   for all 
$  \ 0\le a<b\le T, \ b-a\le \De$ 
\begin{equation}
\begin{cases}\sup_{0\le t\le T} |x(t)|
 \le   c_2 \exp\left\{ c_1 \|B\|_{0, T, \be}^{1/\be}  \right\} \\ \\
\|x\|_{a, b, \be} 
 \le  c_2 \exp\left\{ c_1 \|B\|_{0, T, \be}^{1/\be}  \right\}  \,. 
 \end{cases} \label{e.quasilinear-bound}
\end{equation} 
\end{theorem} 
\begin{proof} The first inequality in \eqref{e.quasilinear-bound} 
is a direct consequence 
of \eqref{e.final-uniform-bound-solution}  and the above second inequality
is the consequence of \eqref{e.final-holder-bound-solution}  together with an easy
bound for $\frac{d}{dt} \La(t)$.  
\end{proof} 

Let us now try to solve    equation   \eqref{e.quasilinear6.6}, which 
 can be written as
\[
z(t)=\eta+\int_0^t \bb(s, z(s), \om)ds +\int_0^t \left[a_1(s, \Ga(s))
z(s) +a_0(s, \Ga(s))\right] \de B(s)\,.
\]
Namely, 
\begin{equation}
z(t)-\int_0^t \left[a_1(s, \Ga(s))
z(s) +a_0(s, \Ga(s))\right] \de B(s)=\eta+\int_0^t \bb(s, z(s), \om)ds \,.
\label{e.6.8}
\end{equation}
Or 
\begin{equation}
dz(t)- \left[a_1(t, \Ga(t))
z(t) +a_0(t, \Ga(t))\right] \de B(s)=  \bb(t, z(t), \om)dt \,.
\label{e.6.9}
\end{equation}
Let 
\begin{equation}
\begin{cases}
 A_1(t)=\exp\left\{-\int_0^t a_1(s, \Ga(s)) \de B(s)\right\} \\
 A_2(t)= \int_0^t A_1(s) \ a_0(s, \Ga(s))  \de B(s)\,.  
\end{cases}\label{e.a1-a2}
\end{equation}
Denote 
\begin{equation}
y(t)=A_1(t)z(t)- A_2(t)\,.
\end{equation}
Then  the equation \eqref{e.6.8} can be written as
\begin{eqnarray}
dy(t)
&=&   A_1(t) \left\{ dz(t)- \left[a_1(t, \Ga(t))
z(t) +a_0(t, \Ga(t))\right] \de B(t)  \right\}\nonumber\\
&=&  A_1(t) \bb(t, z(t), \om) dt\nonumber\\ 
&=&  \cB(t, y(t), \om) dt\,,
\end{eqnarray}
where 
\begin{equation}
 \cB(t, y)=A_1(t)\bb(t,  A_1^{-1}(t) (y+A_2(t)), \om) 
 \label{e.def-cB}
 \end{equation}
with $A_1^{-1}(t)=\exp\left\{\int_0^t a_1(r, \Ga(r)) \de B(r)\right\}$.  
This equation is a (pathwise) ordinary differential equation 
and can be solved by classical method. 
 
 To summarize here is how we can solve the quasilinear equation of  the following form 
 \begin{equation}
 dx(t)=b(t, x(t), \om) dt+\left[a_1(t, \om) x(t)+a_0(t, \om)\right] d B(t)\,,
 \quad x(0)=\eta(\om)\,. 
 \end{equation}
 \begin{enumerate}
 \item[(i)] First we solve the first equation in \eqref{e.quasilinear.transformed}
 to obtain $\Ga(t)$. 
 \item[(ii)] Then we solve \eqref{e.inverse-gamma} to obtain the inverse $\La(t)
 =\La(t,t)$ of
 $\Ga(t)$.   
 \item[(iii)] Define $A_1$ and $A_2$ by \eqref{e.a1-a2}. 
 \item[(iv)]  Define $\cB(t,y)=\cB(t, y, \om)$ by \eqref{e.def-cB} and solve the (ordinary
 differential) equation 
 $\dot y(t)=\cB(t, y(t))\,, \ y(0)=\eta $.  
  \item[(v)]   Let $z(t,\om)= A_1^{-1}(t, \om) (y(t,\om)+A_2(t, \om))$.  
\item[(vi)]  The solution $x(t)$ is then given by 
$x(t,\om)=z(t, \La(t,\om))$.  
 \end{enumerate}
\begin{remark} If $b$ and $\si$ are deterministic,  then the system of equations  
\eqref{e.5.6} becomes  
\begin{equation}
\begin{cases}\Gamma(t)=\om+\int_0^t  \tilde \si  (s, z(s), \Gamma(s))\int_0^\cdot \phi(s,
u)du ds\,;\\  
z(t)= \eta(\om ) +\int_0^t \tilde b(s, z(s) )ds
+\int_0^t \si(s, z(s) )\de B(s)\\ 
\qquad\qquad  +\int_0^t\int_0^s  \si(s, z(s) ) \tilde \si (u, z(u) ) \phi(s, u) duds\,. 
\end{cases} \label{e.6.14a}
\end{equation}
This system of equations is also decoupled. One may first solve the above second equation to obtain $z(t,\om)$  and then substitute it into the above first equation to  obtain  $\Ga(t)$.  However, the main diffculty  remains to study the invertibility of $\Ga(t):\Om\rightarrow \Om$, which is hard as explained in the proof of Theorem \ref{t.5.14}.     
\end{remark}

 \subsection{Linear case} 
 
If $\si$ is linear  as in the  previous subsection and if $b$ is also linear, i.e.
\begin{equation}
\begin{cases}
\si(s, x, \om)=a_1(s, \om) x+a_0(s, \om) \\
b(s,x,\om)=\be_1(s, \om) x+\be_0(s, \om) \,, 
\end{cases} 
\end{equation}
then the equation  is called linear equation.  

$\Gamma$ and $\La$ can be found in the same way as in the 
 quasilinear case.
As we shall see that  $z$ also satisfies a linear equation,
 we explain how to obtain the explicit form for 
$z(t)$.  First,  notice that the second equation in \eqref{e.quasilinear.transformed}
becomes 
\begin{eqnarray}
z(t)
&=&\eta+\int_0^t \be_1 (s, \Ga(s)) z(s) ds+\int_0^t \be_0(s, \Ga(s))ds\nonumber\\
&&\qquad -\int_0^t \left[\DD^\phi_s a_1\right](s, \Ga(s)) z(s)ds -\int_0^t
\left[\DD^\phi_s a_0\right](s, \Ga(s)) ds \nonumber\\
&&\qquad +\int_0^t a_1(s, \Ga(s)) z(s)\de
B(s)+\int_0^t a_0(s, \Ga(s)) \de B(s)\nonumber\\
&&\qquad +\int_0^t a_1(s, \Ga(s))\int_0^s a_1(u, \Ga(u)) \phi(s, u) du z(s) ds \nonumber\\
&&\qquad +\int_0^t a_0(s, \Ga(s))\int_0^s a_1(u, \Ga(u)) \phi(s, u) du   ds\,.
\label{e.linear-z}
\end{eqnarray}
Introduce
\begin{eqnarray}
\Phi(t,s)
&:=&\exp\bigg(\int_s^t \left[ \be_1(u, \Ga(u))-(\DD^\phi_u a_1)(u,
\Ga(u))\right]du
+\int_s^t a_1(u, \Ga(u))\de B(u)\nonumber\\
&&\qquad +\int_s^t a_1(u, \Ga(u))\int_0^u a_1(v, \Ga(v)) \phi(u, v) dv   du \bigg)\,.
\end{eqnarray}
Then the above equation \eref{e.linear-z}  can be solved explicitly as
\begin{eqnarray}
z(t)
&=&\Phi(t, 0) \eta +\int_0^t \Phi(t, s)\left[\be_0(s, \Ga(s))-(\DD^\phi_s
a_0)(s, \Ga(s))\right]ds\nonumber \\
&&\qquad+\int_0^t \Phi(t,s) a_0(s, \Ga(s)) \de B(s)\nonumber \\
&&\qquad +\int_0^t 
\Phi(t,s) a_0(s, \Ga(s)) \int_0^s a_1(v, \Ga(v)) \phi(s, v) dvds\,.
\label{e.linear-z} 
\end{eqnarray}
Thus the solution becomes
\begin{eqnarray}
x(t)
&=& z(t, \La(t))\,. 
\end{eqnarray}
\begin{example}
If $b(s, x,\om)=b(s)x$ and $\si(s,x,\om)=a(s) x$ and $\eta=x_0 $,
where $b(s)$ and $a(s)$ are    deterministic function of $s$, then 
the first equation of \eqref{e.quasilinear.transformed} becomes
\[
\Ga(t)=\om+\int_0^t a (s)\int_0^\cdot \phi(s,u) duds \,.
\]
Thus 
\begin{equation}
\La(t)=\om-\int_0^\cdot h(t, u)du\,, 
\end{equation} 
where 
\begin{equation}
h(t, u)=\int_0^t a (s) \phi(s, u) ds\,. 
\end{equation}
Since $\DD^\phi a  =0$,   we have
\begin{equation}
\Phi(t,s)=\exp\left\{ \int_s^t b (u) ds+\int_s^t a (u) \de B(u)
+\int_s^t a(u)\int_0^u a(v) \phi(u,v) dv
du\right\}\,. 
\end{equation}
Now since $\DD^\phi b=0$ and $\be_0=a_0=0$,  we have
\begin{eqnarray*}
z(t)
&=&\Phi(t, 0) x_0\,.  
\end{eqnarray*} 
Since $\Phi(t,s)=\Phi(t,s, \om)$ still depends on $\om$.  Using lemma \ref{t.3.1}, we have
\begin{eqnarray}
\Psi(t,s)
&:=&\Phi(t, s, \La(t))\nonumber\\
&=& \exp\bigg\{ \int_s^t b(u) ds+\int_s^t a (u) \de B(u)-\int_s^tdu \int_u^t 
a (u)  a (s) \phi(s,u)ds\bigg\} 
 \nonumber \\
\label{e.def-Psi} 
\end{eqnarray}
Thus the solution to 
\begin{equation}
dx(t)=b(t) x(t) dt+ a(t)x(t) dB(t)
\end{equation}
is given by 
\begin{eqnarray}
x(t)
&=&\Psi(t, 0) x_0  \nonumber\\
&=&  \exp\bigg\{ \int_0^t b(u) ds+\int_0^t a (u) \de B(u)\nonumber\\
&&\qquad -\frac12 
\int_0^t  \int_0^t 
a (u)  a (s) \phi(s,u) duds 
\bigg\}  x_0\,.  
\end{eqnarray}
This is well-known (see for example \cite{BHOZ, duncan, humams, hustochastics, huoksendal}).  
\end{example}

\begin{example} In the ame way as above example, we can solve the following 
linear stochastic differential equation
\begin{equation}
dx(t)=[b(t) x(t) +\be(t) ]dt+ [a(t)x(t) +\al(t)] dB(t)\,,\quad x(0)=x_0\,, 
\end{equation}
where $  x_0\in \RR $, $b(t), \be(t), a(t) , \al(t)$ are deterministic functions,
 to obtain
\begin{eqnarray}
x(t)
&=&\Psi(t, 0) x_0+\int_0^t \Psi(t, s) \be (s) ds 
+\int_0^t \Psi(t,s) \al(s ) \de B(s)\nonumber \\
&&\qquad +\int_0^t 
\Psi(t,s) \al (s ) \int_0^s a (v ) \phi(s, v) dvds\,,
\label{e.linear-z-solution} 
\end{eqnarray}
where $\Psi$ is given by \eqref{e.def-Psi}.  
\end{example}

\end{document}